\documentclass[11pt, reqno]{amsart}
\usepackage{pifont}
\usepackage{mathrsfs}
\usepackage{geometry}
\usepackage{titletoc}
\usepackage{stix2}
\usepackage{amsmath}
\usepackage{amssymb} %to produce blackboard bold characters(\supsetneq)
\usepackage{enumitem} %lay­out of the three ba­sic list en­vi­ron­ments
\usepackage{mathtools} %an extension package to amsmath + loading amsmath
\usepackage[table]{xcolor} %color of tables
\usepackage[all]{xy} %xymatrix commutative diagram
\usepackage{tikz} %pictures
\usepackage{tikz-cd}%implies
\usepackage{indentfirst} %Indent first paragraph
\usepackage{babel} %lan­guages
\usepackage{setspace} %set­ting the spac­ing be­tween lines

\usepackage[colorlinks,linkcolor=red,anchorcolor=green,citecolor=blue]{hyperref} %color of hyperlinks
\hypersetup{linktocpage = true} %color of hyperlinks of TOC

\usepackage{rotating} %Rotate an image, table or paragraph
\usepackage{ytableau} %young tableau
\usepackage{longtable} %long table
\newcolumntype{M}[1]{>{\centering\arraybackslash}m{#1}} %define dimension for long stable

\usepackage{fancyhdr}

\newcommand{\norm}[1]{\left\lVert#1\right\rVert}%mathcal letters

\newcommand\sL{{\mathscr L}}

\newcommand{\Rmnum}[1]{\expandafter\@slowromancap\romannumeral #1@}
\theoremstyle{plain}
\newtheorem{thm}{Theorem}[section]
\newtheorem{lem}[thm]{Lemma}
\newtheorem{prop}[thm]{Proposition}
\newtheorem{cor}[thm]{Corollary}
\newtheorem{defn}[thm]{Definition}

\theoremstyle{definition}
\newtheorem{example}[thm]{Example}
\newtheorem{rem}[thm]{Remark}

\newcommand{\btheorem}{\begin{thm}}
    \newcommand{\etheorem}{\end{thm}}
\newcommand{\bproposition}{\begin{prop}}
    \newcommand{\eproposition}{\end{prop}}
\newcommand{\bdefinition}{\begin{defn}}
    \newcommand{\edefinition}{\end{defn}}
\newcommand{\bcorollary}{\begin{cor}}
    \newcommand{\ecorollary}{\end{cor}}
\newcommand{\bproof}{\begin{proof}}
    \newcommand{\eproof}{\end{proof}}
\newcommand{\bremark}{\begin{remark}}
    \newcommand{\eremark}{\end{remark}}
\newcommand{\eexample}{\end{example}}
\newcommand{\bexample}{\begin{example}}

\newcommand{\elemma}{\end{lemma}}
\newcommand{\blemma}{\begin{lemma}}

\newcommand{\la}{\langle}
\newcommand{\ra}{\rangle}

\renewcommand{\bar}{\overline}

\renewcommand{\phi}{\varphi}

\newcommand{\beq}{\begin{equation}}
\newcommand{\eeq}{\end{equation}}
\newcommand{\ee}{\end{eqnarray*}}
\newcommand{\be}{\begin{eqnarray*}}

\newcommand{\bd}{\begin{enumerate}}
    \newcommand{\ed}{\end{enumerate}}

\renewcommand{\tilde}{\widetilde}

\newcommand{\thmref}[1]{Theorem~\ref{#1}}

\newcommand{\lemref}[1]{Lemma~\ref{#1}}

\newcommand{\corref}[1]{Corollory~\ref{#1}}

\newcommand{\R}{{\mathbb R}}
\newcommand{\di}{{\rm div}}

\newcommand{\ric}{\mathrm{Ric}}
\newcommand{\vol}{{\rm{Vol}}}%%%%%%%%%%%%%%%%%%%%%%%%%%%%%

\usepackage{harmony}

\renewcommand{\#}{\sharp}

\newcommand{\lrp}[1]{\left( #1\right)}

\newcommand{\lrb}[1]{\left\{ #1\right\}}

\newcommand{\abs}[1]{\lvert#1\rvert}
\setlist[itemize]{leftmargin=*}
\setlist[enumerate]{leftmargin=*}

\numberwithin{equation}{section} %numbering of equations
\makeatletter

\usepackage{fancyhdr}
\title{Quasi-linear equation $\Delta_pv+av^q=0$ on manifolds with integral bounded Ricci curvature and geometric applications}
\author{Youde Wang}
\author{Guodong Wei}
\author{Liqin Zhang}

\address{Youde Wang, 1. Hua Loo-Keng Key Laboratory of Mathematics, Institute of Mathematics, Academy of Mathematics and Systems Science, Chinese Academy of Sciences, Beijing 100190, China;
		2. School of Mathematical Sciences, University of Chinese Academy of Sciences, Beijing 100049, China.}
\email{wyd@math.ac.cn}
\address{Guodong Wei, School of Mathematics (Zhuhai), Sun Yat-sen University, Zhuhai, Guangdong, 519082, P. R. China}
\email{weigd3@mail.sysu.edu.cn}
\address{Liqin Zhang, School of Mathematics and Information Sciences, Guangzhou University, 510006, P. R. China}
\email{1061837643@qq.com}

\begin{document}
\renewcommand{\subjclassname}{
 \textup{2010} Mathematics Subject Classification}

    \begin{abstract}  
    We study nonexistence results and gradient estimates for solutions of
\[
\Delta_p v + a v^{q}=0
\]
defined on complete Riemannian manifolds satisfying a \emph{$\chi$-type Sobolev inequality}. We establish a Liouville theorem under the assumptions that the underlying manifold $(M,g)$ supports a \emph{$\chi$-type Sobolev inequality} and that the $L^{\frac{\chi}{\chi-1}}$-norm of $\ric_-(x)$ is bounded above by a constant depending only on $\dim(M)$, the Sobolev constant $\mathbb{S}_\chi(M)$, and the volume growth rate of geodesic balls $B_r\subset M$. This extends and improves several recent results of Ciraolo, Farina, and Polvara \cite{CFP}; our approach, however, differs from their ``$P$-function'' method. In addition, for manifolds satisfying a \emph{$\chi$-type Sobolev inequality}, we obtain a lower bound on the volume growth of geodesic balls. We also derive a local logarithmic gradient estimate for positive solutions, assuming that $\ric_-(x)\in L^\gamma$ for some $\gamma > \frac{\chi}{\chi-1}$.

Some geometric and topological applications of our main result are also presented in this article (see \thmref{end}, \thmref{main4}, and \corref{main5}). In particular, we prove the following. Let $(M,g)$ be a complete noncompact Riemannian manifold of dimension $n\ge 3$ on which the Sobolev inequality \eqref{chi-n} holds, and assume that $\ric(x)\ge 0$ outside some geodesic ball $B(o,R_0)$. Then there exists a positive constant $C(n)$, depending only on $n$, such that if
\[
\|\ric_-\|_{L^{\frac{n}{2}}}\leq C(n)\,\mathbb{S}_{\frac{n}{n-2}}(M),
\]
then $(M,g)$ has exactly one end.    \end{abstract}

    \maketitle

    \setcounter{tocdepth}{1}
{\small{    \begin{spacing}{1.08} \tableofcontents %
            % \dottedcontents{chapter}[1.3cm]{}{3em}{5pt}
            \dottedcontents{section}[1cm]{}{2em}{5 pt} %
            % \dottedcontents{subsection}[2cm]{}{3em}{5pt}
\end{spacing} }}
\section{Introduction}
%Gradient estimate is a fundamental and powerful technique in the study of partial differential equations on Riemannian manifolds. For instance, one can use gradient estimates to derive Harnack inequalities, to deduce Liouville type theorems, to study the geometry of manifolds, etc. Many mathematicians pay attention to the study on this topic (see for example,  \cite{MR615628, MR834612,  peng2020gradient, MR2880214, MR431040, MR3866881} and the references there in).
	
In this paper we are concerned with the following quasi-linear equation
\begin{align}\label{equ0}
\Delta_pv +av^{q}=0
\end{align}
defined on a complete Riemannian manifold $(M,g)$ which supports a Sobolev inequality, where $p>1$, $a$, $q\in\R$ are constants, and the $p$-Laplacian operator is defined as
$$\Delta_p(v)=\di(|\nabla v|^{p-2}\nabla v).$$
In the case $a=1$ and $p=2$, equation \eqref{equ0} reduces to the well-known semilinear equation
\begin{equation}\label{se}
\Delta v + v^{q}=0,
\end{equation}
commonly referred to as the Lane-Emden equation. This equation arises in various branches of mathematics, such as the prescribed scalar  curvature problem (for $q=(n+2)/(n-2)$, cf.~e.g. \cite{MR788292, MR929283}), the scalar field equation (cf. \cite{BL}), the stationary solutions to Euler's equation on $\mathbb{S}^2$ (cf.~\cite{CCKW, CG}) and has been studied extensively in the last half century (cf. e.g. \cite{ADW, MR1134481, CM1, CM2, DKW, DN2, GG1, GG2, BJ, GS, lf, PWW1, N, Z}).

The study on the existence and non-existence of positive solutions to the equation \eqref{equ0} and \eqref{se} is rather subtle. It was proved by Gidas and Spruck in \cite{BJ} that any nonnegative solutions to \eqref{se} with $1< q <(n+2)/(n-2)$ on a Riemannian manifold of nonnegative Ricci curvature is zero. Combining the results in \cite{MR4559367, BJ}, we know this result actually holds for $-\infty< q <(n+2)/(n-2)$. On the other hand, Ding-Ni proved in \cite{DN1} that for any $b>0$, there exists a positive solution to \eqref{se} defined in $\mathbb{R}^n$ with $q\geq (n+2)/(n-2)$ such that $\|v\|_{L^\infty}=b$. Moreover, Cafarelli-Gidas-Spruck \cite{CGS} proved that all positive solutions to equation \eqref{se} in $\R^n$ with critical power $q=\frac{n+2}{n-2}$ are radial, and can be explicitly written as
\begin{align}\label{eq:solution-2}
u(x)=\left(a+b\abs{x-x_0}^{2}\right)^{-\frac{n-2}{2}},\quad\quad n\left(n-2\right)ab=1,
\end{align}
where $x_0\in\mathbb{R}^n$, $a>0$ and $b>0$ (cf.~\cite{MR1121147} for the two dimensional case). We refer to \cite{CM1, SunW} and \cite{CFP} for recent development of classification results for equation \eqref{se}.

Now, we turn our attention back to the equation \eqref{equ0}. When $a = 0$, this equation becomes the $p$-Laplacian equation
\begin{align}\label{plaplace}
\Delta_pv=0, \quad p>1.
\end{align}
The renowned Cheng-Yau's logarithm gradient estimate showed that when $p=2$, any solution bounded from above or below to \eqref{plaplace} is a constant provided the Ricci curvature of the Riemannian manifold is nonnegative (cf.~\cite{CY}). Kotschwar-Ni \cite{MR2518892} proved that any positive $p$-harmonic function on a complete Riemannian manifold with nonnegative sectional curvature is a positive constant. Subsequently, this result was verified to remain correct by Wang and Zhang \cite{WZ} under the condition of nonnegative Ricci curvature, where the authors applied the Nash-Moser iteration technique and Saloff-Coste's Sobolev inequality (cf.~\cite[Theorem 3.1]{S}) to study the gradient estimates of equation \eqref{plaplace}. This gradient estimate was later refined to be sharp by Sung and Wang \cite{MR3275651}.
		
When $q\neq p-1$ and  $a>0$, the constant $a$ can be absorbed by a dilation transformation, hence equation (\ref{equ0}) can be reduced to the classical Lane-Emden-Fowler (or Emden-Fowler) equation
\begin{align}\label{equa:1.2}
\Delta_pv +v^{q}=0,
\end{align}
which appears naturally on fluid mechanics and conformal geometry and has been widely studied in the literature (cf.~\cite{BL, MR1004713, MR1134481, HWW, MR829846, O, MR1946918, MR3336621} and the references therein). It was proved by Serrin and Zou in \cite{MR1946918} that if $1<p<n$ and $q>0$, then equation \eqref{equa:1.2} defined on $\R^n$ admits no positive solution if and only if
$$ 0<q<np/(n-p)-1.$$
By employing the Nash-Moser iteration method, He and the first two named authors of the present paper \cite{HWW} showed that there is no positive solution of \eqref{equ0} defined on a complete Riemannian manifold of nonnegative Ricci curvature with
$$a>0 \quad \&\quad  q<\frac{n+3}{n-1}(p-1)\quad \text{or} \quad a<0 \quad \& \quad  q>p-1.$$	
Especially, in the case $p=2$, $a>0$ and $q\in (-\infty,\, \frac{n+2}{n-2})$, Lu \cite{Lu} established the Cheng-Yau type logarithm gradient estimate for positive solutions to Lane-Emden equation \eqref{se} on a complete Riemannian manifold with Ricci curvature bounded from below.

Recently, He, Sun and the first named author of this paper \cite{HSW} showed there is no positive solutions to the subcritical Lane-Emden-Fowler equations(i.e., \eqref{equa:1.2} with
$-\infty< q <\frac{np}{(n-p)_+}-1$),
over complete Riemannian manifolds with nonnegative Ricci curvature, thereby deriving the optimal Liouville theorems for such equations.

Numerous mathematicians have also studied differential inequalities on a complete manifold $(M, g)$, such as
\begin{align}\label{eq:lane-emden-3}
\Delta_pu + u^{q}\leq 0.
\end{align}
Grigor'yan and Sun \cite{GS} and Zhang \cite{Z} investigated the uniqueness of a nonnegative solution to this inequality for $p=2$. Notably, they utilized the condition of volume growth instead of relying on nonnegative Ricci curvature. Similar results hold for general $p>1$ (cf.~\cite{MR3336621}).
%That is, there exists no positive solution to the inequality \eqref{eq:lane-emden-3} provided $q>p-1$ and for some fixed point $o\in M^n$ and some %positive number $C$
%\begin{align*}
%\mathrm{Vol}\left(B_{R}(o)\right)\leq C R^{\frac{pq}{q+1-p}}\ln^{\frac{p-1}{q+1-p}}R,\quad\forall R\gg 1.
%\end{align*}

Very recently, Ciraolo-Farina-Polvara \cite{CFP} studied the Liouville theorem for positive solutions to \eqref{equa:1.2} defined on a manifold associated with a so-called {\it $\chi$-type Sobolev inequality}. In order to introduce their results, we first clarify the definition of this inequality.

\begin{defn}
Let $\left(M^n,g\right)$ be an $n$-dimensional Riemannian manifold. We say the {\it $\chi$-type Sobolev inequality} holds on $(M,g)$, if $\chi>1$ and there exists a positive constant $\mathbb{S}_{\chi}(M)>0$ such that for any $f\in C^{\infty}_0(M,g)$, there holds
\begin{equation}\label{chi}
\mathbb{S}_{\chi}(M)\left(\int_M f^{2\chi}dv\right)^{\frac{1}{\chi}}\leq\int_M |\nabla f|^2 dv.
\end{equation}
\end{defn}

It is well known that \eqref{chi} holds for a large class of Cartan--Hadamard manifolds on which Poincar\'e inequalities are valid; for details, we refer to \cite{HE}. We also note that if a \emph{$\chi$-type Sobolev inequality} holds on $(M,g)$ with $\dim(M)=n\ge 3$, then necessarily $\chi\le \frac{n}{n-2}$ (see Theorem \ref{leq}). Therefore, throughout the remainder of this article we assume $\chi\in\bigl(1,\frac{n}{n-2}\bigr]$ when $\dim(M)\ge 3$.

Clearly, \eqref{chi} reduces to the classical Sobolev inequality when $\chi=\frac{n}{n-2}$, namely,
\begin{equation}\label{chi-n}
\mathbb{S}_{\frac{n}{n-2}}(M)\left(\int_M f^{\frac{2n}{n-2}}\,dv\right)^{\frac{n-2}{n}}
\leq \int_M |\nabla f|^2\,dv,
\end{equation}
which holds for a broad class of complete Riemannian manifolds. In recent years, significant effort has been devoted to optimal Sobolev inequalities on Riemannian manifolds; see, for instance, the survey \cite{DruHeb02AB} and the references therein.

In the sequel, we will use the notation $\ric_{-}(\,\cdot\,)$ which is defined as
$$\ric_-(x)=\max_{|v|=1,\, v\in T_xM}\left\{0,\, -\ric_x(v,v)\right\}, \quad \forall x\in M.$$
%{\color{red} omit? Here $\ric_x: T_xM \to T_xM$ denotes the Ricci curvature tensor at $x\in M$ with the equipped Riemannian metric $g$.}

Recently, Ciraolo, Farina, and Polvara (cf.~\cite{CFP}) proved the following results, which may be viewed as a generalization of the classical Gidas--Spruck's theorem for the semilinear equation $\Delta u + u^q = 0$ on Riemannian manifolds with nonnegative Ricci curvature to the setting where the Ricci curvature is only bounded in an integral sense.

\begin{thm}[Theorem 1.5 in \cite{CFP}]\label{cfp}
Let $(M,g)$ be a complete noncompact Riemannian manifold of dimension $n\geq3$ on which the \it{$\chi$-type Sobolev inequality} holds, and let $u$ be a nonnegative solution to
\begin{equation*}
\Delta u +u^q=0\ \ \ \ \text{in} \ \  M,
\end{equation*}
with $1<q<(n+2)/(n-2)$. Assume that
\begin{equation}\label{ric-}
\|\ric_-\|_{L^{\frac{\chi}{\chi-1}}}\leq\frac{1}{48}\left(\frac{n+2}{n-2}-q\right)(q-1)(n-2)\mathbb{S}_{\chi}(M)^2.
\end{equation}
Assume also that, for some fixed point $o\in M$, the volume of the geodesic ball $B(o,R)$ satisfies
\begin{equation}\label{vol1}
\vol(B(o,R))=O\left(R^{2+\frac{8}{q-1}}\right),\ \ \ \text{as}\ \ R\rightarrow \infty.
\end{equation}
Then $u$ is identical to zero.
\end{thm}

It was first observed by the first-named author \cite[Proposition 2.4]{W}, and later independently by Carron (cf.~\cite{C}) and Akutagawa (cf.~\cite{AK}), that if a Riemannian manifold $(M^n,g)$ has positive Sobolev constant, then the volume of any geodesic ball of radius $R$ satisfies
\[
\vol(B_R)\ge C R^n,
\]
where the constant $C$ depends only on $n$ and the Sobolev constant.

The first main result of the present article shows that an analogous volume growth estimate also holds for Riemannian manifolds satisfying the \emph{$\chi$-type Sobolev inequality}. Throughout the paper, we use the abbreviation $B_r$ to denote a geodesic ball of radius $r$ centered at an arbitrary point of the manifold. We now state the result precisely as follows.
\begin{thm}\label{ve}
Let $(M,g)$ be a complete noncompact Riemannian manifold on which the {$\chi$-type Sobolev inequality} holds. Then,
$$\mathrm{Vol}\left(B_r\right)\geq C\left(\chi, \mathbb{S}_{\chi}(M)\right)r^{\frac{2\chi}{\chi-1}}.$$
\end{thm}

It is worth pointing out that $\frac{2\chi}{\chi-1}=n$ when $\chi=\frac{n}{n-2}$, where $n=\dim(M)$. In this case, we recover the volume estimate obtained in \cite{W}. Moreover, it is easy to see that
\[
\frac{2\chi}{\chi-1}>2+\frac{8}{q-1}
\]
whenever $q>4\chi-3$. Consequently, for such $q$ there is a structural incompatibility between the assumptions that a \emph{$\chi$-type Sobolev inequality} holds and that the volume growth condition \eqref{vol1} is satisfied in the theorem of Ciraolo--Farina--Polvara. The main contribution of the present paper is to remove the \emph{a priori} volume growth assumption from that theorem. We now state our result precisely.

\begin{thm}\label{main}
Let $(M,g)$ be a complete noncompact Riemannian manifold on which the {\it $\chi$-type Sobolev inequality} holds. Assume that  $\vol(B(o,R))= O(R^{\beta^*})$ for some $\beta^*\geq n$ in the case $\chi=\frac{n}{n-2}$, where $B(o,R)\subset M$ is a geodesic ball centered at a fixed point $o\in M$. Then there exists a positive constant $C(n,p,q,\beta^*)$ depending on $n$, $p$, $q$ and $\beta^*$ such that, if
$$a=0\quad\text{or} \quad a>0\quad  \&\quad  q<\frac{n+3}{n-1}(p-1)\quad \text{or} \quad a<0 \quad \& \quad  q>p-1$$
and
$$\|\ric_-\|_{L^{\frac{\chi}{\chi-1}}}\leq C(n,p,q,\beta^*)\mathbb{S}_{\chi}(M),$$
then equation \eqref{equ0} does not admit any positive solution for $a\neq 0$ and does not admit any non-constant positive solution for $a=0$.
\end{thm}

Roughly speaking, we show that if $(M,g)$ satisfies a \emph{$\chi$-type Sobolev inequality} and has volume growth strictly subexponential, then, provided the $L^{\chi/(\chi-1)}$-norm of $\ric_-$ is bounded by $C'\,\mathbb{S}_{\chi}(M)$, where $C'$ depends on the growth rate of the volume of geodesic balls, equation \eqref{equ0} admits no positive solutions.

Moreover, in the case $p=2$ and $a=1$, we obtain the following conclusion:

\begin{thm}\label{theorem0}
Let $(M,~g)$ be a complete noncompact Riemannian manifold on which the $\chi$-type Sobolev inequality holds. Assume that $\mathrm{Vol}\left(B(o,R) \right)=O(R^{\beta^*})$ for some $\beta^*\geq n$ in the case $\chi=\frac{n}{n-2}$, where $B(o,R)\subset M$ is a geodesic ball centered at a fixed point $o\in M$. Then there exists a positive constant $C(n,q,\beta^*)$ depending on $n$, $q$ and $\beta^*$ such that, if
\begin{align*}
q<\frac{n+2}{(n-2)_+}
\end{align*}
and
\begin{align*}
\|\mathrm{Ric}_-\|_{L^\frac{\chi}{\chi-1}}\le C(n,q,\beta^*)\mathbb{S}_\chi(M),
\end{align*}
then Lane-Emden equation \eqref{se} does not admit any positive solution.
\end{thm}

\begin{rem}
Note that \(\frac{\chi}{\chi-1}>\frac{n}{2}\) whenever \(\chi\in\bigl(1,\,\frac{n}{n-2}\bigr)\).
By the relative volume comparison theorem under integral Ricci curvature bounds
due to Petersen and Wei~\cite{PW}, if \(\|\ric_-\|_{L^{\frac{\chi}{\chi-1}}}\) is bounded, then
\[
\mathrm{Vol}\bigl(B(o,R)\bigr)=O\!\left(R^{\frac{2\chi}{\chi-1}}\right)\quad \text{as}\quad R\rightarrow\infty.
\]
Consequently, the assumption `` $\mathrm{Vol}\bigl(B(o,R)\bigr)=O(R^{\beta^*})$ for some
$\beta^*\geq n$" in Theorem~\ref{main} and Theorem~\ref{theorem0} is needed only in
the borderline case \(\chi=\frac{n}{n-2}\).
\end{rem}

As a direct corollary of \thmref{main}, we obtain the following Liouville type result for $p$ harmonic functions.

\begin{cor}\label{main2}
Let $(M,g)$ be a complete noncompact Riemannian manifold on which the Sobolev inequality \eqref{chi-n} holds true. Assume that $\dim(M)=n\geq 3$ and $\vol(B(o,R))= O(R^{\beta^*})$ for some $\beta^*\geq n$ where $B(o,R)$ is a geodesic ball centered at fixed point $o\in M$. Then there exists a positive constant $C(n, p, \beta^*)$ depending on $n$, $p$ and $\beta^*$ such that, if
$$\|\ric_-\|_{L^{\frac{n}{2}}}\leq C(n,p,\beta^*)\mathbb{S}_{\frac{n}{n-2}}(M),$$
then there is no nonconstant, positive $p$-harmonic function on $(M, g)$.
\end{cor}

\begin{rem}
Here, we record several remarks on \thmref{main} and \thmref{theorem0}.

\bd
  \item Our results remove the prior volume growth assumption \eqref{vol1}.

  \item In \thmref{cfp}, the restriction $q>1$ is required due to the curvature assumption \eqref{ric-}. In contrast, our results remain valid for all $q\in(-\infty,1]$.

  \item We establish a Liouville theorem for the general $p$-Laplace equation \eqref{equ0}, where $a$ is an arbitrary constant.

  \item In a certain sense, these two theorems may be viewed as effective versions of the classical Liouville results of Gidas--Spruck~\cite{BJ} and Serrin--Zou~\cite{MR1946918}.
  To explain this, let $g_\varepsilon$ be a family of metrics on $\mathbb{R}^n$ such that $g_\varepsilon$ agrees with the Euclidean metric outside a compact set and $g_\varepsilon \to g_{\mathrm{Euc}}$ in $C^{2}$ as $\varepsilon\to 0$.
  The Positive Mass Theorem implies that $g_\varepsilon$ cannot have nonnegative Ricci curvature.
  Since the Sobolev constant depends only on the $C^{0}$-geometry of the metric, our results show that the classical Liouville property still holds for such a deformation provided that
  \[
  \bigl\|\ric^{g_\varepsilon}_{-}\bigr\|_{L^{\frac{n}{2}}}
  \le C(n,p,q)\,\mathbb{S}_{\frac{n}{n-2}}(\mathbb{R}^n,g_\varepsilon).
  \]
\ed
\end{rem}
%It is worth noticing that $$\|\ric_-\|_{L^{\frac{n}{2}}}=\left(\int_M\ric_-^\frac{n}{2}(x)dM\right)^\frac{2}{n}$$ is scale invariant with respect to the metric $g$ equipped on $M$, so the quantity is of obvious and important geometric significance.}
Assume that $\ric_-(x)\in L^{\gamma}(B_1)$ for some $\gamma>\frac{\chi}{\chi-1}$. By the relative volume comparison theorem under integral Ricci curvature bounds due to Petersen and Wei (cf.~\cite{PW}), we obtain the following local gradient estimate for solutions of \eqref{equ0}.

\begin{thm}\label{main3}
Let $(M,g)$ be a complete noncompact Riemannian manifold with $\dim(M)\geq 3$ on which the {\it $\chi$-type Sobolev inequality} holds. Assume $v$ is a positive solution to \eqref{equ0} on $B_1\subset M$. Then, when
$$a=0\quad\text{or} \quad a>0\quad  \&\quad  q<\frac{n+3}{n-1}(p-1)\quad \text{or} \quad a<0 \quad \& \quad  q>p-1,$$
the following gradient estimate holds true
\begin{equation}\label{1017}
\sup_{B_{1/2}}\frac{|\nabla v|^2}{v^2}\leq C\left(p, q, \mathbb{S}_{\chi}(M), \gamma, \|\ric_-\|_{L^{\gamma}(B_1)}\right),
\end{equation}
where $\gamma$ is any number greater than $\chi/(\chi-1)$. In particular, if $\|\ric_-\|_{L^{\gamma}(M)}\leq\Lambda$, then the following global gradient estimate holds
$$\frac{|\nabla v|^2}{v^2}\leq C\left(p, q, \mathbb{S}_{\chi}(M), \gamma, \Lambda\right).$$
\end{thm}

In \cite[Theorem 1.2 ]{PW1}, Petersen and Wei proved that if a manifold has positive Sobolev constant and $\ric_-\in L^p$ for some $p>\frac{n}{2}$, then any positive harmonic function $v$ on $B_1$ satisfies the local gradient estimate
\[
\sup_{B_{1/2}} |\nabla v| \le C \sup_{B_1} v.
\]
Therefore, Theorem \ref{main3} both improves and extends the result of Petersen and Wei.

Now, we turn to recalling some topological features of noncompact complete manifolds, with a focus on estimates for the number of ends. Notice that ``{\it an end $E$ of a manifold $M$ is an unbounded component of the complement of some compact subset $D$ of $M$. In this case,
we say that $E$ is an end corresponding to $D$}".

For a noncompact complete manifold whose Ricci curvature is nonnegative outside a compact set, Cai~\cite{Cai} proved that such a manifold has only finitely many ends. Later, in 1995, Cai, Colding, and Yang~\cite{CCY} studied a gap phenomenon for the ends of this class of manifolds. More precisely, they proved the following theorem:
\begin{quotation}
Given $n > 0$, there exists an $\epsilon(n) > 0$ such that for all pointed open complete manifolds $(M^n , o)$ with Ricci curvature bounded from below by $-(n-1)\Lambda^2$ and nonnegative outside the ball $B(o, a)$, if $\Lambda a < \epsilon(n)$,
then $M^n$ has at most two ends.
\end{quotation}
For more results related to this topic, we refer to \cite{CM, Wu} and references therein.

On the other hand, using harmonic function theory to study the number of ends has a long history (see for example, \cite{CSZ, DW, LTam}). Especially, if the Sobolev constant of $(M,g)$ is positive and $\dim(M)\geq 3$, the first named author has ever used harmonic function theory to prove that such a manifold is of finitely many ends and the dimension of linear space spanned by bounded harmonic functions on $(M, g)$ equals the number of ends (cf.~\cite{W}).

By the volume estimate established in \thmref{ve} we can employ the method adopted in \cite{DW} and \cite{W} to show the following theorem.

\begin{thm}\label{end}
Let $(M, g)$ be a complete noncompact Riemannian manifold with Sobolev constant $\mathbb{S}_{\chi}(M)>0$. If $(M, g)$ has at least $k$ ends, then $\dim(H^\infty(M))\geq k$.
\end{thm}

As a geometric application of \corref{main2} and \thmref{end}, we obtain the following gap theorem for ends.

\begin{thm}\label{main4}
Let $(M,g)$ be a complete noncompact Riemannian manifold on which the Sobolev inequality \eqref{chi-n} holds. Assume that $\dim(M)=n\geq 3$ and $\vol(B(o,R))= O(R^{\beta^*})$ for some $\beta^*\geq n$ where $B(o,R)$ is a geodesic ball centered at fixed point $o\in M$.Then there exists a positive constant $C(n,\beta^*)$ depending only on $n$ and $\beta^*$ such that, if
$$\|\ric_-\|_{L^{\frac{n}{2}}}\leq C(n,\beta^*)\mathbb{S}_{\frac{n}{n-2}}(M),$$
then $(M, g)$ is of a unique end.
\end{thm}

As a direct corollary of the above theorem, we have the following:
\begin{cor}\label{main5}
Let $(M,g)$ be a complete noncompact Riemannian manifold with $\dim(M)=n\geq 3$ on which the Sobolev inequality \eqref{chi-n} holds true. Assume Ricci curvature is nonnegative outside some compact set, or more generally, $\mathrm{Vol}\left(B_r\right)=O(r^n)$ for any $r\to\infty$. Then there exists a positive constant $C(n)$ depending only on $n$ such that, if
$$\|\ric_-\|_{L^{\frac{n}{2}}}\leq C(n)\mathbb{S}_{\frac{n}{n-2}}(M),$$
then $(M, g)$ is of a unique end. In particular, if $(M,g)$ is a complete Riemannian manifold with $\dim(M)=n\geq 3$ and nonnegative Ricci curvature on which the Sobolev inequality \eqref{chi-n} holds true, then $(M, g)$ has only an end.
\end{cor}

%It should be pointed out that the conclusion in the above corollary ``a complete Riemannian manifold with $\dim(M)=n\geq 3$ and nonnegative Ricci curvature, on which the Sobolev inequality \eqref{chi-n} holds true, has only an end" is implied by Theorem 3.3 in \cite{W}.

Combining Theorem \ref{main4} and Cai-Colding-Yang's result, we would like to ask the following problem: {\em whether or not there exists a positive constant $\epsilon(n)$ depending on $n$ such that any noncompact complete Riemannian manifold with nonnegative Ricci curvature outside a compact set is of at most two ends if
$$\|\ric_-\|_{L^{\frac{n}{2}}}\leq \epsilon(n);$$
or more generally,  does there exist \(\epsilon(n)>0\) such that any noncompact complete Riemannian manifold has at most two ends whenever}
\[
\|\ric_-\|_{L^{\frac{n}{2}}}\le \epsilon(n)?
\]

The rest of the paper is organized as follows. In Section~2, we derive a volume estimate for geodesic balls $B_R$ in $(M,g)$ under the assumption that $(M,g)$ satisfies a \emph{$\chi$-type Sobolev inequality}. Section~3 is devoted to a careful estimate of $\sL\!\left(|\nabla \log v|^{2\alpha}\right)$ (the operator $\sL$ is defined explicitly in \eqref{linea}). Our proofs show that, by choosing a suitable parameter $\alpha$, one can obtain effective integral estimates for the gradient of positive solutions to \eqref{equ0}. In particular, under the assumptions of Theorem~\ref{main} or Theorem~\ref{main3}, we obtain an $L^{\theta\chi}$ bound for $|\nabla \log v|$. Crucially, this bound depends only on the radius and the volume of the geodesic ball. With this integral bound in hand, we then apply the Nash--Moser iteration scheme to complete the proof of \thmref{main3}. In Section~4, we proceed to prove \thmref{theorem0} by choosing appropriate auxiliary functions. Finally, in Section~5 we prove \thmref{end} and \thmref{main4}.

\section{Volume estimate}
In this section, we shall provide two different proofs of \thmref{ve}. One of the two proofs is to make use of the Nash-Moser iteration initially developed by Wang in \cite{W}, the other is a direct iteration of volume of the geodesic ball as in \cite{HE}.

First, we will show the following result.

\begin{thm}\label{leq}
Let $(M^n,g)$ be a complete noncompact Riemannian manifold on which the {\it $\chi$-type Sobolev inequality} holds with $n\geq3$, then $\chi\leq n/(n-2)$.\end{thm}
\begin{proof}
Fix a point $o\in M$ and fix some $r>0$. Define
\begin{equation*}
u(x)=\left\{
\begin{aligned}
r-d_g(x,0)&,\ \ \ \ \text{if}\ \ d_g(x,o)\leq r,\\
0&,\ \ \ \ \text{if}\ \ d_g(x,0)\geq r,
\end{aligned}
\right.
\end{equation*}
where $d_g(\cdot,\cdot)$ is the distance function on $(M,g)$. Obviously, $u\in W_0^{1,2}(M,g)$. Notice that
$$\int_M |\nabla u|^2=\mathrm{Vol}\left(B_r(o)\right).$$

On the other hand, we know that
$$\mathrm{Vol}\left(B_r(o)\right)=\omega_nr^n(1+o(1)),\quad\mathrm{and}\quad \mathrm{Area}\left(\partial B_r(o)\right)=n\omega_nr^{n-1}(1+o(1))\ \ \ \text{as}\ r\rightarrow 0.$$
Direct calculation shows that
$$\left(\int_M u^{2\chi}\right)^{\frac{1}{\chi}}=\left(n\omega_n\frac{\Gamma(n)\Gamma(2\chi+1)}{\Gamma(n+2\chi+1)}\right)^{\frac{1}{\chi}}r^{2+\frac{n}{\chi}}(1+o(1))\ \ \text{as}\ r\rightarrow 0.$$
From the definition of the $\chi$-type Sobolev inequality, there must hold
$$n\leq 2+\frac{n}{\chi}.$$
Hence, $\chi\leq n/(n-2)$.
\end{proof}

Now, we will establish a local maximum principle for subharmonic functions on Riemannian manifolds on which the $\chi$-type Sobolev inequality holds via the classical Nash-Moser iteration. Then, we will see later that \thmref{ve} is a direct corollary of this local maximum principle.

\begin{lem}\label{nm1}
Let $(M,g)$ be a complete noncompact Riemannian manifold on which the $\chi$-type Sobolev inequality holds. Assume $u\in W^{1,2}\left(B_r\right)$ satisfying
\begin{equation*}
\Delta u\geq0,
\end{equation*}
in the weak sense, i.e.,
\begin{equation*}\int_{B_r} \langle\nabla u,\nabla \phi\rangle\leq 0,\ \ \text{for any}\ \ 0\leq\phi\in C^{\infty}_0\left(B_r\right).\end{equation*}
Then, for any $s>0$ and $0<\theta<1$, there holds
\begin{equation}\label{nm2}
\sup_{B_{\theta r}}u\leq C\left(\chi,s,\theta, \mathbb{S}_\chi(M)^{-1}\right)r^{-\frac{2\chi}{s(\chi-1)}}\left(\int_{B_r}(u^+)^s\right)^{1/s}.
\end{equation}
\end{lem}

\begin{proof}
Since $u^+$ is also the subsolution, without loss of generality we assume $u\geq0$. We first prove this lemma in the case $s\geq2$. By integration by part, we know
$$\int_{B_r} \langle\nabla u,\nabla \phi\rangle\leq 0,\ \ \text{for any}\ \ 0\leq\phi\in W^{1,2}_0\left(B_r\right).$$
For any $\eta(x)\in C_0^{\infty}\left(B_r\right)$, substituting $\phi=\eta^2u^{s-1}$ into the above inequality yields
\begin{equation}\label{eq1}
(s-1)\int_{B_r}|\nabla u|^2u^{s-2}\eta^2\leq-2\int_{B_r}\eta u^{s-1}\langle\nabla u,\nabla \eta\rangle.
\end{equation}
By Young's inequality, we deduce that
\begin{equation}\label{eq2}
\int_{B_r}|\nabla u|^2u^{s-2}\eta^2\leq\frac{4}{(s-1)^2}\int_{B_r}u^s|\nabla \eta|^2.
\end{equation}
Note that
$$u^{s-2}|\nabla u|^2=\frac{4}{s^2}|\nabla u^{\frac{s}{2}}|^2.$$
We know
$$|\nabla (\eta u^{s/2})|^2=\eta^2|\nabla u^{\frac{s}{2}}|^2+u^s|\nabla \eta|^2+s\eta u^{s-1}\langle\nabla u,\nabla \eta\rangle.$$
This implies
\begin{eqnarray}\label{eq4}
\int_{B_r}|\nabla (\eta u^{s/2})|^2 &=&\int\eta^2|\nabla u^{\frac{s}{2}}|^2+\int u^s|\nabla \eta|^2+s\int\eta u^{s-1}\langle\nabla u,\nabla \eta\rangle\nonumber\\
&\leq& \frac{s^2}{(s-1)^2} \int u^s|\nabla \eta|^2+\int u^s|\nabla \eta|^2+\int u^s|\nabla \eta|^2+\frac{s^2}{4}\int u^{s-2}|\nabla u|^2\eta ^2\nonumber\\
&\leq&2\left(1+\frac{s^2}{(s-1)^2}\right)\int u^s|\nabla \eta|^2\nonumber\\
&\leq&10\int u^s|\nabla \eta|^2.
\end{eqnarray}
By Sobolev inequality, there holds
$$\mathbb{S}_{\chi}(M)\left(\int_{B_r}(\eta u^{s/2})^{2\chi}\right)^{\frac{1}{\chi}}\leq\int_{B_r}|\nabla (\eta u^{s/2})|^2.$$
Combining the above inequality with \eqref{eq4}, we arrive at
\begin{equation}\label{eq3}
\left(\int_{B_r}(\eta u^{s/2})^{2\chi}\right)^{\frac{1}{\chi}}\leq 10\mathbb{S}_\chi(M)^{-1}\int u^s|\nabla \eta|^2.
\end{equation}

For some positive number $\theta\in(0,1)$, let
$$r_k=r\lrp{\theta+\frac{1-\theta}{2^k}},\ \ \ k=0, 1, 2, \cdots$$
and choose $\eta_k\in C_0^{\infty}\left(B_{r_k}\right)$ such that $\eta_k\equiv 1$ on $B_{r_{k+1}}$ and
\begin{equation}\label{eq5}
|\nabla \eta_k|\leq\frac{2}{r_k-r_{k+1}}=\frac{2^{k+2}}{(1-\theta)r}.
\end{equation}
Let
$$s_k=s\chi^{k}.$$
Substituting $\eta\triangleq\eta_k$, $s\triangleq s_k$ and \eqref{eq5} into \eqref{eq3}, we obtain
\begin{equation}\label{eq6}
\|u\|_{L^{s_{k+1}}\lrp{B_{r_{k+1}}}}\leq\lrb{\frac{160\times4^k}{\mathbb{S}_\chi(M)(1-\theta)^2r^2}}^{\frac{1}{s_k}} \|u\|_{L^{s_{k}}\lrp{B_{r_{k}}}}.
\end{equation}
By iteration, we derive
\beq\label{eq7}
\|u\|_{L^{s_{k+1}}\lrp{B_{r_{k+1}}}}\leq4^{\sum_{i=0}^{k}\frac{i}{s_i}}\left\{\frac{160}{\mathbb{S}_\chi(M)(1-\theta)^2r^2}\right\}^{\sum_{i=0}^{k}\frac{1}{s_i}} \|u\|_{L^{s}\lrp{B_{r}}}.
\eeq
Since
$$\sum_{i=0}^{\infty}\frac{1}{s_i}=\frac{\chi}{s(\chi-1)}\quad\mbox{and}\quad \sum_{i=0}^{\infty}\frac{i}{s_i}=\frac{\chi}{s(\chi-1)^2}.$$
Letting $k\rightarrow\infty$ in \eqref{eq7} yields
\begin{equation}\label{eq8}
\|u\|_{L^{\infty}\left(B_{\theta r}\right)}\leq C\left(\chi,s,\mathbb{S}_\chi(M)^{-1}\right)[(1-\theta)r]^{-\frac{2\chi}{s(\chi-1)}}\|u\|_{L^s\left(B_r\right)}.
\end{equation}

Next, we prove Lemma \ref{nm1} when $s\in(0,2)$. By letting $s=2$ in \eqref{eq8}, we obtain
\begin{equation*}
\sup_{B_{\theta r}}u\leq C\left(\chi,\mathbb{S}_\chi(M)^{-1}\right)[(1-\theta)r]^{-\frac{\chi}{\chi-1}}\|u\|_{L^2\left(B_r\right)}.\end{equation*}
Hence, for $0<s<2$, there holds true
\begin{equation*}
\sup_{B_{\theta r}}u\leq C\left(\chi,\mathbb{S}_\chi(M)^{-1}\right)[(1-\theta)r]^{-\frac{\chi}{\chi-1}}\left(\sup_{B_r}u\right)^{1-\frac{s}{2}}
\left(\int_{B_r}u^s\right)^{\frac{1}{2}}.
\end{equation*}
By Young's inequality, we deduce that
\begin{equation}\label{eq9}
\sup_{B_{\theta r}}u\leq \frac{2-s}{2}\sup_{B_r}u+\frac{s}{2}C\left(\chi,\mathbb{S}_\chi(M)^{-1}\right)^{\frac{2}{s}}[(1-\theta)r]^{-\frac{2\chi}{s(\chi-1)}}
\|u\|_{L^s\left(B_r\right)}.
\end{equation}
Let $\tilde{s}=\theta r$, $t=r$, $\psi(s)=\sup_{B_{\theta r}}u$ and $\psi(t)=\sup_{B_{r}}u$. Then \eqref{eq9} can be rewritten as
\begin{equation*}
\psi(\tilde{s})\leq\frac{2-s}{2}\psi(t)+C\left(\chi,\mathbb{S}_\chi(M)^{-1}\right)(t-\tilde{s})^{-\frac{2\chi}{s(\chi-1)}}\|u\|_{L^s\left(B_r\right)}.\end{equation*}
By Lemma \ref{hl} below, we conclude that for $s\in(0,\,2)$, inequality \eqref{eq8} also holds. Thus, we complete the proof.
\end{proof}

\begin{lem}[cf. \cite{ChWu}]\label{hl}
Let $f(t)\geq0$, $t\in[\tau_0,\tau_1]$ with $\tau_0\geq0$. Suppose for $\tau_0\leq t<\tilde{s}\leq\tau_1$,
$$f(t)\leq\theta f(\tilde{s})+\frac{A}{(\tilde{s}-t)^{\alpha}}+B$$
for some $\theta\in[0,1)$. Then for any $\tau_0\leq t<\tilde{s}\leq\tau_1$, there holds
$$f(t)\leq c(\alpha,\theta)\left(\frac{A}{(\tilde{s}-t)^\alpha}+B\right).$$
\end{lem}
\begin{proof}[Proof of \thmref{ve}(Method 1):]   It is easy to see that Theorem \ref{ve} can be directly  deduced from Lemma \ref{nm1} by letting $u=1$, $s=1$ and $\theta=\frac{1}{2}$ in \eqref{nm2}.
\end{proof}

Next, we shall prove Theorem \ref{ve} via a direct iteration of the volume of the geodesic balls as in \cite{HE}.
\begin{proof}[Proof of Theorem \ref{ve}:] Recall that the $\chi$-type Sobolev inequality tells us that for any $u\in W^{1,2}_0(M)$, there holds
$$\mathbb{S}_{\chi}(M)\left(\int_M u^{2\chi}dv\right)^{\frac{1}{\chi}}\leq\int_M |\nabla u|^2.$$
Now, let $r>0$ and $x$ be some point of $M$, and let $u\in W^{1,2}_0(M)$ be such that $u=0$ on $M\backslash B_x(r)$. By H\"older's inequality, we have
$$\lrp{\int_M u^2dv}^{\frac{1}{2}}\leq \lrp{\int_M u^{2\chi}dv}^{\frac{1}{2\chi}}\vol\lrp{B_x(r)}^{\frac{\chi-1}{2\chi}}.$$
Hence,
\begin{eqnarray}\label{218}
\frac{\lrp{\int_M |\nabla u|^2dv}^{\frac{1}{2}}}{\lrp{\int_M u^2dv}^{\frac{1}{2}}}&\geq&\sqrt{\mathbb{S}_{\chi}(M)}\frac{\lrp{\int_M u^{2\chi}dv}^{\frac{1}{2\chi}}}{\lrp{\int_M u^2dv}^{\frac{1}{2}}}\nonumber\\
&\geq&\frac{\sqrt{\mathbb{S}_{\chi}(M)}}{\vol\lrp{B_x(r)}^{\frac{\chi-1}{2\chi}}}.
\end{eqnarray}
From now on, let
\begin{equation*}
u(y)=\left\{
\begin{aligned}
r-d_g(x,y)&,\ \ \ \ \text{if}\ \ d_g(x,y)\leq r,\\
0&,\ \ \ \ \text{if}\ \ d_g(x,y)\geq r,
\end{aligned}
\right.
\end{equation*}
where $d_g(\cdot,\cdot)$ is the distance function on $(M,g)$. Obviously, $u$ is Lipschitz and $u=0$ on $M\backslash B_x(r)$. Substituting $u$ into \eqref{218} yields
\begin{eqnarray*}
\frac{\int_M |\nabla u|^2dv}{\int_M u^2dv}&=&\frac{\vol\left(B_x(r)\right)}{\int_{B_x(r)} u^2dv}\\
&\geq&\frac{\mathbb{S}_{\chi}(M)}{\vol\left(B_x(r)\right)^{\frac{\chi-1}{\chi}}}.
\end{eqnarray*}
Note that
$$\int_{B_x(r)} u^2dv\geq \int_{B_x\left(r/2\right)} u^2dv,$$
and
$$ \int_{B_x\left(r/2\right)} u^2dv\geq \frac{r^2}{2^2}\vol\left(B_x\left(\frac{r}{2}\right)\right).$$
It is easy to see that
$$\frac{\mathbb{S}_{\chi}(M)}{\vol\left(B_x(r)\right)^{\frac{\chi-1}{\chi}}}\leq\frac{\vol\left(B_x(r)\right)}{\int_{B_x\left(r/2\right)} u^2dv}\leq \frac{2^2\vol\left(B_x(r)\right)}{r^2\vol\left(B_x\left(\frac{r}{2}\right)\right)}.$$
Hence, we conclude from the above inequality that
$$\vol\left(B_x(r)\right)\geq\left(\frac{r\sqrt{\mathbb{S}_{\chi}(M)}}{2}\right)^{\frac{2\chi}{2\chi-1}}\vol\left(B_x\left(\frac{r}{2}\right)\right)^{\frac{\chi}{2\chi-1}}.$$
Hence, for any $m\in \mathbb{N}$, there holds
$$\vol\left(B_x\left(\frac{r}{2^m}\right)\right)\geq\left(r\sqrt{\mathbb{S}_{\chi}(M)}\right)^{\frac{2\chi}{2\chi-1}}{2}^{-\frac{2(m+1)\chi}{2\chi-1}}\vol\left(B_x\left(\frac{r}{2^{m+1}}\right)\right)^{\frac{\chi}{2\chi-1}}.$$
By induction, we then arrive at
\begin{equation}\label{220}\vol\left(B_x(r)\right)\geq\left(r\sqrt{\mathbb{S}_{\chi}(M)}\right)^{2\alpha(m)}{2}^{-2\beta(m)}\vol\left(B_x\left(\frac{r}{2^{m}}\right)\right)^{\gamma(m)},\end{equation}
where
$$\alpha(m)=\sum^{m}_{i=1}\left(\frac{\chi}{2\chi-1}\right)^i, \ \ \beta(m)=\sum^{m}_{i=1}i\left(\frac{\chi}{2\chi-1}\right)^i,$$
and
$$\gamma(m)=\left(\frac{\chi}{2\chi-1}\right)^m.$$
Direct computation shows that
$$\lim_{m\rightarrow\infty}\alpha(m)=\frac{\chi}{\chi-1}\quad\quad \text{and}\quad\quad \lim_{m\rightarrow\infty}\beta(m)=\frac{\chi(2\chi-1)}{(\chi-1)^2}.$$

On the other hand, it's well known that the volume of geodesic ball with radius $r$ has the following expansion (cf.~\cite{GHL})
$$\vol\left(B_x(r)\right)=b_nr^n\left(1-\frac{\mathrm{R}_g(x)}{6(n+2)}r^2+o\left(r^2\right)\right),$$
where $\mathrm{R}_g(x)$ denotes the scalar curvature of $(M,g)$ at $x$ and $b_n$ is the volume of the Euclidean ball of radius one. Hence,
$$\lim_{m\rightarrow \infty}\vol\left(B_x\left(\frac{r}{2^{m}}\right)\right)^{\gamma(m)}=1.$$
By letting $m\rightarrow\infty$, we obtain
$$\vol\left(B_x(r)\right)\geq C\left(\chi,\mathbb{S}_\chi(M)\right)r^{\frac{2\chi}{\chi-1}}.$$
Thus we complete the proof of Theorem \ref{ve}.
\end{proof}

\section{$p$-Laplace case: Proof of \thmref{main} and \thmref{main3}}
\subsection{Linearization operator $\sL$ of $p$-Laplacian.}\
\\
Recall that the $p$-Laplace operator is defined as
\begin{equation}\label{lnz}
\Delta_pu=\di\left(|\nabla u|^{p-2}\nabla u\right).
\end{equation}
The solution of $p$-Laplace equation $\Delta_pu=0$, usually called $p$-harmonic function, is the critical point of the energy functional
$$
E(u)=\int_M|\nabla u|^p.
$$
From the definition, we see that a $2$-harmonic function is just a usual harmonic function.
	
\begin{defn}\label{def1}
$v$ is said to be a (weak) solution of equation (\ref{equ0}) on a region $\Omega\subset M$, if $v\in L^\infty_{loc}(\Omega)\cap W^{1,p}_{loc}(\Omega)$ and for all $\psi\in W^{1,p}_0(\Omega)$, we have
\begin{align*}
-\int_\Omega|\nabla v|^{p-2}\la\nabla v,\nabla\psi\ra +\int_\Omega av^q\psi=0.
\end{align*}
\end{defn}

From now on, we always assume that $v\in W_{loc}^{1,p}\left(\Omega\right)\cap L_{loc}^{\infty}\left(\Omega\right)$ is a weak and positive solution of the equation \eqref{equ0}. We denote
$$
\Omega_{cr} = \{x\in \Omega:\nabla v(x)=0\}.
$$
According to Theorem 1.4 in \cite{ACF} and the classical regularity theory (for example, see \cite{MR0709038, MR1946918, MR0727034, MR0474389}), we know that
$$v\in C_{loc}^{1,\beta}(\Omega)\cap W^{2,2}_{loc}(\Omega\setminus\Omega_{cr})\quad \mbox{and}\quad v\in C_{loc}^{\infty}(\Omega^c_{cr}).$$
	
On the other hand, it is easy to see from \eqref{lnz} that  the linearization operator $\sL$ of the $p$-Laplace operator is
\begin{align*}
\sL(\psi)=\di\left(|\nabla u|^{p-2}A(\nabla \psi)\right),
\end{align*}
where
\begin{align*}
A(\nabla\psi) = \nabla\psi+(p-2)|\nabla u|^{-2}\la\nabla \psi,\nabla u\ra\nabla u.
\end{align*}
	
Now, let $v$ be a positive solution to equation \eqref{equ0}. By a logarithmic transformation
$$u = -(p-1)\log v,$$
equation (\ref{equ0}) becomes
\begin{align}\label{equ21}
\Delta_pu -|\nabla u|^p-be^{cu}=0,
\end{align}
where
\begin{align*}
b = a(p-1)^{p-1},\quad  c = \frac{p-q-1}{p-1}.
\end{align*}
Denote $f = |\nabla u|^2$. Then the linearization operator $\sL$ of the $p$-Laplace operator can be rewritten as
\begin{align}\label{linea}
\sL(\psi)=\di\left(f^{p/2-1}A(\nabla \psi)\right),
\end{align}
with
\begin{align}\label{defofA}
A(\nabla\psi) = \nabla\psi+(p-2)f^{-1}\la\nabla \psi,\nabla u\ra\nabla u.
\end{align}
Next, we calculate the explicit expression $\sL(f^\alpha)$ for any $\alpha>0$ that will play a key role in our proof.
	
\begin{lem}
For any $\alpha >0$, the equality
\begin{align}\label{bochner1}
\begin{split}
\sL (f^{\alpha}) = &\alpha\left(\alpha+\frac{p}{2}-2\right)f^{\alpha+\frac{p}{2}-3}|\nabla f|^2+2\alpha f^{\alpha+\frac{p}{2}-2} \left(|\nabla\nabla u|^2 + \ric(\nabla u,\nabla u) \right)\\
&+\alpha(p-2)(\alpha-1)f^{\alpha+\frac{p}{2}-4}\langle\nabla f,\nabla u\rangle^2 + 2\alpha f^{\alpha-1}\langle\nabla\Delta_p u,\nabla u\rangle
\end{split}
\end{align}
holds point-wisely in $\{x:f(x)>0\}$.
\end{lem}
	
\begin{proof}
By the definition of $A$ in (\ref{defofA}), we have
$$A\Big(\nabla (f^{\alpha})\Big) = \alpha f^{\alpha-1}\nabla f + \alpha(p-2)f^{\alpha-2}\langle\nabla f,\nabla u\rangle \nabla u = \alpha f^{\alpha-1}A(\nabla f).$$
Hence
\begin{align*}
\sL(f^{\alpha}) & = \alpha \text{div}\Big( f^{\alpha-1}f^{\frac{p}{2}-1}A(\nabla f)\Big)= \alpha \Big\langle\nabla( f^{\alpha-1}),f^{\frac{p}{2}-1}A(\nabla f)\Big\rangle + \alpha f^{\alpha-1}\mathcal{L} (f).
\end{align*}
A straightforward computation shows that
\begin{align}\label{equ2.6}
\alpha \Big\langle\nabla( f^{\alpha-1}),f^{\frac{p}{2}-1}A(\nabla f)\Big\rangle=&\Big\langle \alpha(\alpha-1)f^{\alpha-2}\nabla f, f^{\frac{p}{2}-1}\nabla f + (p-2)f^{\frac{p}{2}-2}\langle\nabla f,\nabla u\rangle\nabla u\Big\rangle,
\end{align}
and
\begin{align}\label{equ2.7}
\begin{split}
\alpha f^{\alpha-1}\sL (f)=& \alpha f^{\alpha-1}\Big(\left(\frac{p}{2}-1\right)f^{\frac{p}{2}-2}|\nabla f|^2 + f^{\frac{p}{2}-1}\Delta f + (p-2)\left(\frac{p}{2}-2\right)f^{\frac{p}{2}-3}\langle\nabla f,\nabla u\rangle^2\\
&+ (p-2)f^{\frac{p}{2}-2}\langle\nabla\langle\nabla f,\nabla u\rangle,\nabla u\rangle + (p-2)f^{\frac{p}{2}-2}\langle\nabla f,\nabla u\rangle\Delta u\Big).
\end{split}
\end{align}
Combining \eqref{equ2.6} and \eqref{equ2.7} yields
\begin{align}\label{equ2.9}
\begin{split}
\sL(f^\alpha)=&\alpha\left(\alpha+\frac{p}{2}-2\right)f^{\alpha+\frac{p}{2}-3}|\nabla f|^2 + \alpha f^{\alpha+\frac{p}{2}-2}\Delta f\\
&+\alpha(p-2)\left(\alpha+\frac{p}{2}-3\right)f^{\alpha+\frac{p}{2}-4}\langle\nabla f,\nabla u\rangle^2\\
&+ \alpha(p-2)f^{\alpha+\frac{p}{2}-3}\langle\nabla\langle\nabla f,\nabla u\rangle,\nabla u\rangle +\alpha(p-2)f^{\alpha+\frac{p}{2}-3}\langle\nabla f,\nabla u\rangle\Delta u.
\end{split}
\end{align}
		
Notice that, by the definition of the $p$-Laplacian we have
\begin{align*}
\langle\nabla\Delta_p u,\nabla u\rangle &= \left(\frac{p}{2}-1\right)\left(\frac{p}{2}-2\right)f^{\frac{p}{2}-3}\langle\nabla f,\nabla u\rangle^2 + \left(\frac{p}{2}-1\right)f^{\frac{p}{2}-2}\langle\nabla\langle\nabla f,\nabla u\rangle,\nabla u\rangle\\
&\quad+ \left(\frac{p}{2}-1\right)f^{\frac{p}{2}-2}\langle\nabla f,\nabla u\rangle\Delta u + f^{\frac{p}{2}-1}\langle\nabla\Delta u,\nabla u\rangle.
\end{align*}
Hence, the last term of the right hand side of (\ref{equ2.9}) can be rewritten as
		\begin{align}\label{last}
			\begin{split}
				\alpha(p-2)f^{\alpha+\frac{p}{2}-3}\langle\nabla f,\nabla u\rangle\Delta u
				&
				=
				2\alpha f^{\alpha-1}\langle\nabla\Delta_p u,\nabla u\rangle
				- 2\alpha f^{\alpha+\frac{p}{2}-2}\langle\nabla\Delta u,\nabla u\rangle
				\\
				&
				\quad- \alpha(p-2)\left(\frac{p}{2}-2\right)f^{\alpha+\frac{p}{2}-4}\langle\nabla f,\nabla u\rangle^2
				\\
				&
				\quad- \alpha(p-2)f^{\alpha+\frac{p}{2}-3}\langle\nabla\langle\nabla f,\nabla u\rangle,\nabla u\rangle.
			\end{split}
		\end{align}
Moreover, the Bochner formula tells us that
$$\langle\nabla\Delta u,\nabla u\rangle=\frac{1}{2}\Delta f - |\nabla\nabla u|^2 - \ric(\nabla u,\nabla u).$$
By substituting the above and \eqref{last} into \eqref{equ2.9}, we finally arrive at
\begin{align*}
\sL(f^{\alpha}) = &\alpha\left(\alpha+\frac{p}{2}-2\right)f^{\alpha+\frac{p}{2}-3}|\nabla f|^2 + 2\alpha f^{\alpha+\frac{p}{2}-2} \left(|\nabla\nabla u|^2 + \ric(\nabla u,\nabla u) \right)\\
&+\alpha(p-2)(\alpha-1)f^{\alpha+\frac{p}{2}-4}\langle\nabla f,\nabla u\rangle^2 +2\alpha f^{\alpha-1}\langle\nabla\Delta_p u,\nabla u\rangle.
\end{align*}
We finish the proof of the lemma.
\end{proof}

\subsection{Precise estimate of $\sL$.}\
\\
In this subsection, we shall  prove a precise estimate for $\sL(f^\alpha)$ when $v$ is a positive solution to equation \eqref{equ0}.
\begin{lem}\label{lem31}
Let $u$ be a solution of equation (\ref{equ21}) on $(M,g)$.  Denote
$$f=|\nabla u|^2 \ \ \ \text{and}\ \ \ a_1=\left|p-\frac{2(p-1)}{n-1}\right|.$$ Then the following holds point-wisely in $\{x\in M:f(x)>0\}$:
\begin{align}\label{n1}
\sL(f^{\alpha})\geq 2\alpha f^{\alpha+\frac{p}{2}-2}\left(\beta_{n,p,q ,\alpha}  f^{2 } - \ric_- f -\frac{a_1}{2}  f^{ \frac{1}{2}}|\nabla f|\right),
\end{align}
provided
\begin{enumerate}
\item $\alpha\in[1,\infty)$ and $\beta_{n,p,q ,\alpha} =1/(n-1)$ when $$a\left(\frac{n+1}{n-1}-\frac{q}{p-1}\right)\geq0,$$ 			
			
\item $\alpha>\alpha_0$, where
$$\alpha_0(n,p,q)=\frac{\frac{4}{n-1}+(p-n)\left(\frac{n+1}{n-1}-\frac{q}{p-1}\right)^2}{2\left(\frac{4}{n-1}-(n-1)
(\frac{n+1}{n-1}-\frac{q}{p-1})^2\right)},$$
and
\begin{align*}
\beta_{n,p,q,\alpha}=\frac{1}{n-1}-\left(\frac{n+1}{n-1}-\frac{q}{p-1}\right)^2\frac{ (2\alpha-1)(n-1)+p-1  }{4(2\alpha-1)}>0
\end{align*}
\noindent{when}
\begin{align*}
p-1<q<\frac{n+3}{n-1}(p-1).
\end{align*}	
\end{enumerate}
\end{lem}

\begin{proof}
Let $\{e_1,e_2,\ldots, e_n\}$ be  an orthonormal frame of $TM$ on a domain of  $\{x\in M:f(x)>0\}$ such that $\nabla u=|\nabla u|e_1$. Under this frame, there holds $$u_1=|\nabla u|=f^{1/2},\quad \text{and}\quad u_i=0 \quad \text{for} \ \ 2\leq i\leq n.$$
Moreover, notice that
\begin{eqnarray*}
2fu_{11}&=&2|\nabla u|^2\nabla^2u(e_1,e_1)\\
%&=&2|\nabla u|^2\left(e_1|\nabla u|-\nabla_{e_1}e_1(u)\right)\\
%&=&2|\nabla u|^2\langle e_1,\nabla|\nabla u|\rangle\\
&=&\langle \nabla u,\nabla f\rangle,\end{eqnarray*}
and
\begin{eqnarray*}
|\nabla f|^2&=&\sum_{i=1}^n|2u_1u_{1i}|^2=4f\sum_{i=1}^nu_{1i}^2.
\end{eqnarray*}
Hence,
\begin{equation}\label{equ:3.1}
u_{11} = \frac{1}{2}f^{-1}\la\nabla u,\nabla f\ra\ \ \ \text{and}\ \ \ \frac{|\nabla f|^2}{f}=4\sum_{i=1}^n u_{1i}^2.
\end{equation}
Meanwhile, $\Delta_p u$ has the following expression (cf. \cite{HWW, MR2518892}),
\begin{align*}
\Delta_p u
=&f^{\frac{p}{2}-1}\left((p-1)u_{11}+\sum_{i=2}^nu_{ii}\right).
\end{align*}
Substituting the above equality into equation (\ref{equ21}), we obtain:
\begin{align}\label{equ3.3}
(p-1)u_{11}+\sum_{i=2}^nu_{ii}=f+be^{cu}f^{1-\frac{p}{2}}.
\end{align}
By Cauchy inequality, we arrive at
\begin{align}\label{equ3.5}
|\nabla\nabla u|^2\geq& \sum_{i=1}^nu_{1i}^2 + \sum_{i=2}u_{ii}^2 \geq \frac{|\nabla f|^2}{4f} +\frac{1}{n-1}\left(\sum_{i=2}u_{ii}\right)^2.
\end{align}
It follows from \eqref{equ21} that
\begin{align*}
\langle\nabla\Delta_p u,\nabla u\rangle =  pf^{\frac{p}{2} }u_{11}+bce^{cu}f.
\end{align*}
Substituting \eqref{equ:3.1}, (\ref{equ3.5}) and the above equality into (\ref{bochner1}) yields
\begin{align}\label{equ3.60}
\begin{split}
\frac{f^{2-\alpha-\frac{p}{2}}}{2\alpha} \sL \left(f^{\alpha}\right)\geq &\frac{1}{2} \left(\alpha+\frac{p-3}{2}\right) \frac{|\nabla f|^2}{f}
+\frac{1}{n-1}\left(\sum_{i=2}u_{ii}\right)^2+ \ric(\nabla u, \nabla u)\\
&+2(p-2)(\alpha-1)u_{11}^2 + f^{1-\frac{p}{2}}\left(pf^{\frac{p}{2} }u_{11}+bce^{cu}f\right).
\end{split}
\end{align}
Furthermore, by the facts that
$$\frac{|\nabla f|^2}{f}\geq 4u_{11}^2,\ \ \ \text{and}\ \ \ \alpha+\frac{p-3}{2}>0,$$
we can infer from \eqref{equ3.60} that
\begin{align}\label{equ3.6}
\begin{split}
\frac{f^{2-\alpha-\frac{p}{2}}}{2\alpha} \sL \left(f^{\alpha}\right) \geq & 2\left(\alpha+\frac{p-3}{2}\right) u_{11}^2 +\frac{1}{n-1}\left(\sum_{i=2}^nu_{ii}\right)^2+ \ric(\nabla u, \nabla u)\\
&+2(p-2)(\alpha-1)u_{11}^2 + f^{1-\frac{p}{2}}\left(pf^{\frac{p}{2} }u_{11}+bce^{cu}f\right).
\end{split}
\end{align}
By \eqref{equ3.3}, we have
\begin{align*}
\left(\sum_{i=2}^nu_{ii}\right)^2 =&\left(f+be^{cu}f^{1-\frac{p}{2}}-(p-1)u_{11}\right)^2\\
=& f^2+\left(be^{cu}f^{1-\frac{p}{2}}-(p-1)u_{11}\right)^2+2be^{cu}f^{2-\frac{p}{2}}-2f(p-1)u_{11}.
\end{align*}
Substituting the above inequality into (\ref{equ3.6}) yields
\begin{align}\label{equ3.7}
\begin{split}
\frac{f^{2-\alpha-\frac{p}{2}}}{2\alpha} \sL\left(f^\alpha\right)\geq &(p-1)(2\alpha-1)u_{11}^2-\ric_- f + \left(p-\frac{2(p-1)}{n-1}\right)f u_{11}+ \frac{f^2}{n-1}\\
&+b\left(c+\frac{2}{n-1}\right)e^{cu}f^{2-\frac{p}{2}}+\frac{1}{n-1}   \left(be^{cu}f^{1-\frac{p}{2}}-(p-1)u_{11}\right)^2.
\end{split}
\end{align}
		
Now, denote
$$a_1 = \left|p-\frac{2(p-1)}{n-1}\right|.$$
It follows from (\ref{equ:3.1}) that
\begin{align*}
2\left(p-\frac{2(p-1)}{n-1}\right)f u_{11}\geq -a_1f^{\frac{ 1}{2} }|\nabla f|.
\end{align*}
Hence,
\begin{align}\label{equa:3.7}
\begin{split}
\frac{f^{2-\alpha-\frac{p}{2}}}{2\alpha} \sL(f^\alpha)\geq& (p-1)(2\alpha-1)u_{11}^2-\ric_- f-\frac{a_1}{2}f^{\frac{ 1}{2} }|\nabla f|+ \frac{f^2}{n-1}\\
&+b\left(c+\frac{2}{n-1}\right)e^{cu}f^{2-\frac{p}{2}}
+\frac{1}{n-1}\left(  be^{cu}f^{1-\frac{p}{2}}-(p-1)u_{11}\right)^2.
\end{split}
\end{align}
		
\noindent\textbf{Case $1$:} the constants $a$, $p$ and $q$ satisfy $$a\left(\frac{n+1}{n-1}-\frac{q}{p-1}\right)\geq0.$$
For this case we have
$$
be^{cu}f\left(c+\frac{2}{n-1}\right)=a(p-1)^{p-1}e^{cu}f\left(\frac{n+1}{n-1}-\frac{q}{p-1}\right)\geq 0.
$$
Since $\alpha\geq 1$, by discarding some non-negative terms in (\ref{equa:3.7}), we obtain
\begin{align*}
\sL(f^\alpha)\geq&  2\alpha f^{\alpha+\frac{p}{2}-2}\left(\frac{  f^{2}}{n-1}- \ric_- f -\frac{a_1}{2}f^{\frac{1}{2} }|\nabla f|\right),
\end{align*}
which is just the inequality in the first case of \lemref{lem31}.
		
By expanding the last term of the right hand side of (\ref{equa:3.7}), we obtain
\begin{align}\label{2.12}
\begin{split}
\frac{f^{2-\alpha-\frac{p}{2}}}{2\alpha} \sL (f^{\alpha})\geq &(p-1)\left(2\alpha-1+\frac{p-1}{n-1}\right)u_{11}^2 - \ric_- f\\
&+b\left(\frac{n+1}{n-1}-\frac{q}{p-1}\right)e^{cu}f^{2-\frac{p}{2}}-\frac{a_1}{2}f^{\frac{ 1}{2} }|\nabla f|+\frac{f^2}{n-1}\\
&+\frac{1}{n-1}\left(  b^2e^{2cu}f^{2-p}-2(p-1)be^{cu}f^{1-\frac{p}{2}}u_{11}\right).
\end{split}
\end{align}
		
\noindent\textbf{Case $2$ :} the constants $a$, $p$ and $q$ satisfy
\begin{align*}
p-1<q<\frac{n+3}{n-1}(p-1).
\end{align*}
In the present situation, the condition is equivalent to
\begin{align*}
\frac{1}{n-1}-\frac{n-1}{4}\left(\frac{n+1}{n-1}-\frac{q}{p-1}\right)^2>0.
\end{align*}
This implies
\begin{align*}
\lim_{\alpha\to\infty}\frac{1}{n-1}-\left(\frac{n+1}{n-1}-\frac{q}{p-1}\right)^2\frac{(2\alpha-1)(n-1)+p-1}{4(2\alpha-1)}>0.
\end{align*}
		
By the monotonicity of the left hand side of the above with repect to $\alpha$, it's easy to see that if we denote
$$\alpha_0(n,p,q)=\frac{\frac{4}{n-1}+(p-n)\left(\frac{n+1}{n-1}-\frac{q}{p-1}\right)^2}{2\left(\frac{4}{n-1}-(n-1)\left(\frac{n+1}{n-1}
-\frac{q}{p-1}\right)^2\right)},$$
then $\beta_{n,p,q,\alpha}>0$ when $\alpha>\alpha_0$.

On the other hand, by using the inequality $\lambda^2-2\lambda\mu\geq -\mu^2$ we have
\begin{align}\label{equa:2.13}
\begin{split}
&(p-1)\left(2\alpha-1+\frac{p-1}{n-1}\right)u_{11}^2-2\frac{(p-1)}{n-1}be^{cu}f^{1-\frac{p}{2}}u_{11}\\
\geq &-\frac{(p-1)b^2e^{2cu}f^{2-p}}{((2\alpha-1)(n-1)+p-1)(n-1)} .
\end{split}
\end{align}
Combining (\ref{2.12}) and (\ref{equa:2.13}) yields
\begin{align}\label{equa:2.14}
\begin{split}
\frac{f^{2-\alpha-\frac{p}{2}}}{2\alpha} \sL(f^\alpha)\geq &\frac{(2\alpha-1) b^2e^{2cu}f^{2-p}}{(2\alpha-1)(n-1)+p-1}-\ric_-f -\frac{a_1}{2}f^{\frac{1}{2}}|\nabla f|\\
& + b\left(\frac{n+1}{n-1}-\frac{q}{p-1}\right)e^{cu}f^{2-\frac{p}{2}}+\frac{f^2}{n-1}.
\end{split}
\end{align}
Applying the relation $\lambda^2+2\lambda\mu\geq -\mu^2$ again, we have
\begin{align}\label{equa:3.15}
\begin{split}
&\frac{(2\alpha-1) b^2e^{2cu}f^{2-p}}{ (2\alpha-1)(n-1)+p-1} + b\left(\frac{n+1}{n-1}-\frac{q}{p-1}\right)e^{cu}f^{2-\frac{p}{2}}\\
\geq& -\left(\frac{n+1}{n-1}-\frac{q}{p-1}\right)^2\frac{ (2\alpha-1)(n-1)+p-1  }{4(2\alpha-1) }f^2.
\end{split}
\end{align}
Substituting (\ref{equa:3.15}) into (\ref{equa:2.14}), we arrive at
\begin{align*}
\frac{f^{2-\alpha-\frac{p}{2}}}{2\alpha} \sL (f^\alpha)\geq&\left(\frac{1}{n-1}-\left(\frac{n+1}{n-1}-\frac{q}{p-1}\right)^2
\frac{(2\alpha-1)(n-1)+p-1}{4(2\alpha-1) }\right)f^2\\
&- \ric_- f-\frac{a_1}{2}f^{\frac{1}{2}}|\nabla f|.
\end{align*}
Hence,
\begin{align*}
\sL(f^\alpha)\geq&2\beta_{n,p,q,\alpha} \alpha f^{\alpha+\frac{p}{2} } - 2\alpha\ric_{-}f^{\alpha+\frac{p}{2}-1}
-a_1\alpha f^{\alpha+\frac{p}{2}-\frac{3}{2}}|\nabla f| ,
\end{align*}
where $\beta_{n,p,q,\alpha}>0$ is defined in Lemma \ref{lem31}. Thus, we complete the proof of this lemma.
\end{proof}

	%Denote $\alpha_\gamma=\alpha_0+\gamma$ and $\beta_{n,p,q,\gamma} = \beta_{n,p,q,\alpha_\gamma}$. Fom the definition of $\beta_{n,p,q,\alpha}$ (see \eqref{defofdelta}), it's easy to see that
	%\begin{align*}
		%\beta_{n,p,q} < \frac{1 }{n-1}.
	%\end{align*}
	%So, we actually have proved that there holds
	%\begin{align}
		%\label{equa3.14}
		%\mathcal L(f^{\alpha_\gamma})\geq
		%2\alpha_\gamma f^{\alpha_\gamma+\frac{p}{2}-2}\left(\beta_{n,p,q ,\gamma}  f^{2 } - \ric_- f
		%-\frac{a_1}{2}  f^{ \frac{1}{2}}|\nabla f|\right),
	%\end{align}
	%if one of the conditions (1) and (2) in \lemref{linear} is satisfied.	

\subsection{Approximation procedure and key integral inequality}\

Now, we are going to establish a key integral inequality of $f=|\nabla u|^2$.
\begin{lem}\label{41}
Let $\Omega = B_R(o)\subset M$ be a geodesic ball. Define $\alpha$ and $\beta_{n,p,q,\alpha}$ as in \lemref{lem31}. Denote
$$\theta(\alpha,t,p)=\alpha+t+\frac{p}{2}-1.$$
Then, for \begin{equation}\label{tc}
t\in\left(\frac{a_1^2}{a_2\beta_{n,p,q,\alpha}},\,\, \infty\right),
\end{equation}
the following inequality holds true
\begin{equation}\label{924}
\frac{\beta_{n,p,q,\alpha}}{2} \int_\Omega f^{\theta+1}\eta^2 + \frac{a_2t}{\theta^2}\mathbb{S}_{\chi}(M) \left\|f^{\theta} \eta^2\right\|_{L^{\chi}}
\leq \int_\Omega \ric_-f^{\theta}\eta^2+\gamma(\alpha,p,t)\int_\Omega f^{\theta}|\nabla\eta|^2,
\end{equation}	
where
$$a_1 = \left|p-\frac{2(p-1)}{n-1}\right|,\ \ \ a_2=\min\{1,\, p-1\}, \ \ \ \text{and}\ \ \ \gamma(\alpha,p,t)=\frac{2(p+1)^2 }{a_2t}+\frac{a_2t}{\theta^2}.$$	
\end{lem}
	
\begin{proof} Since \lemref{lem31} only holds pointwisely on $\{x\in M: f(x)>0\}$. In order to obtain this integral estimate, we need to perform an approximation procedure as that in \cite{WZ, HWW}. Now, let $\eta\in C^{\infty}_0(\Omega,\R)$ be a non-negative and smooth function on $\Omega$ with compact support. Denote $f_\epsilon=(f-\epsilon)^+$. By multiplying $\psi = f_{\epsilon}^{t}\eta^2$ on both side of \eqref{n1}(where $t>1$ is to be determined later), we have
\begin{align*}
&-\int_\Omega\la f^{p/2-1}\nabla f^{\alpha} +(p-2)f^{p/2-2}\la\nabla f^{\alpha},\nabla u\ra\nabla u,\,\nabla \psi \ra\\
\geq &2\beta_{n,p,q,\alpha }\alpha\int_\Omega f^{\alpha+\frac{p}{2} }f^{t}_\epsilon\eta^2  -2\alpha\int_\Omega \ric_-f^{\alpha+\frac{p}{2} -1}f^{t}_\epsilon\eta^2 - a_1\alpha\int_\Omega f^{\alpha+\frac{p-3}{2} }f^{t}_\epsilon|\nabla f|\eta^2.
\end{align*}
Hence,
\begin{align}\label{equa:3.17}
\begin{split}
&-\int_\Omega (\alpha tf^{\alpha+\frac{p}{2}-2}f^{t-1}_\epsilon|\nabla f|^2\eta^2+ t\alpha(p-2)f^{\alpha+\frac{p}{2}-3}f^{t-1}_\epsilon\la\nabla f,\nabla u\ra^2\eta^2)\\
&-\int_\Omega(2\eta\alpha f^{\alpha+\frac{p}{2}-2}f^{t}_\epsilon\la\nabla f,\nabla\eta\ra +2\alpha\eta(p-2)f^{\alpha +\frac{p}{2}-3}f^{t}_\epsilon\la\nabla f,\nabla u\ra\la\nabla u, \nabla\eta\ra)\\
\geq & 2\beta_{n,p,q,\alpha }\alpha\int_\Omega f^{\alpha+\frac{p}{2} }f^{t}_\epsilon\eta^2  -2\alpha\int_\Omega \ric_-f^{\alpha+\frac{p}{2} -1}f^{t}_\epsilon\eta^2- a_1\alpha\int_\Omega f^{\alpha+\frac{p-3}{2} }f^{t}_\epsilon|\nabla f|\eta^2.
\end{split}
\end{align}
Notice that
\begin{align}\label{2.24}
f^{t-1}_\epsilon |\nabla f|^2 +(p-2)f^{t-1}_\epsilon f^{-1}\la\nabla f,\nabla u\ra^2\geq a_2 f^{t-1}_\epsilon|\nabla f|^2,
\end{align}
where $a_2 = \min\{1,\, p-1\}$ and
\begin{align}\label{2.25}
f^{t}_\epsilon\la\nabla f,\nabla\eta\ra+ (p-2)f^{t}_\epsilon f^{-1}\la\nabla f,\nabla u\ra\la\nabla u,  \nabla\eta\ra\geq -(p+1)f^{t}_\epsilon |\nabla f||\nabla\eta|.
\end{align}
Denote
$$\theta=\alpha+t+\frac{p}{2}-1.$$
Substituting (\ref{2.24}) and (\ref{2.25}) into (\ref{equa:3.17}), and then letting $\epsilon\to0$, we arrive at
\begin{align}\label{2.26}
\begin{split}
&2\beta_{n,p,q,\alpha }\int_\Omega f^{\theta+1}\eta^2 + a_2  t\int_\Omega f^{\theta-2}|\nabla f|^2\eta^2\\
\leq &2\int_\Omega \ric_-f^{\theta}\eta^2 + a_1 \int_\Omega f^{\theta-\frac{1}{2}}|\nabla f|\eta^2+2 (p+1)\int_\Omega  f^{\theta-1}|\nabla f||\nabla\eta|\eta.
\end{split}
\end{align}
Since $u\in W^{2,2}_{loc}(\Omega\setminus\Omega_{cr})\cap C^{1,\beta}(\Omega)$ and the measure of critical set $\Omega_{cr}$ is zero by a very recent result \cite[Corollary 1.6]{ACF}, we have $f\in C^\beta(\Omega)$ and $|\nabla f|\in L^2_{loc}$, and hence the integrals in the above make sense.
		
By Cauchy-inequality, we have
\begin{align}\label{2.27}
a_1 f^{\theta-\frac{1}{2}}|\nabla f|\eta^2\leq \frac{a_2t}{4}f^{\theta-2}|\nabla f|^2\eta^2 +\frac{a_1^2}{a_2t} f^{\theta+1}\eta^2,
\end{align}
and
\begin{align}\label{new}
2(p+1)f^{\theta-1}|\nabla f||\nabla\eta|\eta\leq \frac{a_2t}{4}f^{\theta-2}|\nabla f|^2\eta^2
+\frac{4(p+1)^2 }{a_2t} f^{\theta}|\nabla \eta|^2.
\end{align}
Combining the fact
$$t\in\left(\frac{a_1^2}{a_2\beta_{n,p,q,\alpha}},\,\, \infty\right),$$
we conclude by substituting \eqref{2.27} and \eqref{new} into \eqref{2.26} that
\begin{equation}\label{2.29}
\beta_{n,p,q ,\alpha}\int_\Omega f^{\theta+1}\eta^2 + \frac{a_2t}{2}\int_\Omega f^{\theta-2}|\nabla f|^2\eta^2
\leq 2\int_\Omega \ric_-f^{\theta}\eta^2 +\frac{4(p+1)^2 }{a_2t}  \int_\Omega f^{\theta}|\nabla \eta|^2.
\end{equation}
		
On the other hand,
\begin{align}\label{2.30}
\begin{split}
\frac{1}{2}\left|\nabla \left(f^{\frac{\theta}{2}}\eta \right)\right|^2\leq & \left|\nabla f^{\frac{\theta}{2}}\right|^2\eta^2 +f^{\theta}|\nabla\eta|^2\\
=&\frac{\theta^2}{4}f^{\theta-2}|\nabla f |^2\eta^2 +f^{\theta}|\nabla\eta|^2  .
\end{split}
\end{align}
Substituting (\ref{2.30}) into (\ref{2.29}) yields
\begin{align}\label{n2}
&\beta_{n,p,q,\alpha } \int_\Omega f^{\theta+1}\eta^2 + \frac{2a_2t}{\theta^2}\int_\Omega\left|\nabla \left(f^{\frac{\theta}{2} }\eta\right)\right|^2 \\ \nonumber
\leq &2\int_\Omega \ric_-f^{\theta}\eta^2 + \left(\frac{4(p+1)^2 }{a_2t}+\frac{2a_2t}{\theta^2}\right) \int_\Omega f^{\theta}|\nabla\eta|^2.
\end{align}
By Sobolev inequality, there holds
$$\mathbb{S}_{\chi}(M)\left\|f^{\frac{\theta}{2}}\eta\right\|_{L^{2\chi}(\Omega)}^2\leq \int_{\Omega}\left|\nabla \left(f^{\frac{\theta}{2}}\eta\right)\right|^2.$$
Hence,
\begin{align}\label{3.32}
\begin{split}
&\frac{\beta_{n,p,q,\alpha}}{2} \int_\Omega f^{\theta+1}\eta^2 + \frac{a_2t}{\theta^2}\mathbb{S}_{\chi}(M) \left\|f^{\theta} \eta^2\right\|_{L^{\chi}}\\
\leq &\int_\Omega \ric_-f^{\theta}\eta^2+\left(\frac{2(p+1)^2 }{a_2t}+\frac{a_2t}{\theta^2}\right)\int_\Omega f^{\theta}|\nabla\eta|^2.
\end{split}
\end{align}
We complete the proof.
\end{proof}		

\subsection{Local $L^{\theta\chi}$ bound of the gradient.} \label{sec3.3}\

Next, we are ready to show the following $L^{\theta \chi}$ bound with
$$\theta=\alpha+t+p/2-1$$
of the gradient of positive solutions to equation \eqref{equ0}.
\begin{lem}\label{lpbound}
Let $(M,g)$ be a complete manifold on which the $\chi$-type Sobolev inequality holds. Assume $u$ is a positive solution to equation \eqref{equ21} on the geodesic ball $B(o,R)\subset M$. Let $f=|\nabla u|^2$ and denote $\theta$ by
$$\theta = \alpha+t+\frac{p}{2}-1,$$
where $\alpha$ and $t$ satisfy the conditions in \lemref{lem31} and \eqref{tc} respectively. Assume further that
$$\beta(\alpha,p,t)\triangleq\frac{a_2t}{\theta^2}\mathbb{S}_{\chi}(M)-\left\|\ric_-\right\|_{L^{\frac{\chi}{\chi-1}}}>0.$$
Then there exists $a_3 = a_3(n,\alpha,p,q,t)>0$ such that
\begin{align}\label{lpbpund}
\|f \|_{L^{\theta\chi}(B_{3R/4}(o))}\leq a_3\left(\frac{V}{R^{2(\theta+1)}}\right)^{\frac{1}{\theta}},
\end{align}
where $V$ denotes the volume of geodesic ball $B_R(o)$.
\end{lem}
	
\begin{proof}
By H\"older's inequality, we have
\begin{equation*} \int_\Omega \ric_-f^{\theta}\eta^2\leq \left\|\ric_-\right\|_{L^{\frac{\chi}{\chi-1}}}\left\|f^{\theta}\eta^2\right\|_{L^{\chi}}.
\end{equation*}		
Substituting this into \lemref{41}, we conclude that
\begin{align}\label{924}
&\frac{\beta_{n,p,q,\alpha}}{2} \int_{B(o,R)} f^{\theta+1}\eta^2 + \left(\frac{a_2t}{\theta^2}\mathbb{S}_{\chi}(M)- \left\|\ric_-\right\|_{L^{\frac{\chi}{\chi-1}}}\right)\left\|f^{\theta} \eta^2\right\|_{L^{\chi}}\\ \nonumber
\leq &\gamma(\alpha,p,t)\int_{B(o,R)} f^{\theta}|\nabla\eta|^2,
\end{align}	
where
$$\gamma(\alpha,p,t)=\frac{2(p+1)^2 }{a_2t}+\frac{a_2t}{\theta^2}.$$

Now, choose  $\eta_1\in C^{\infty}_0(B_R(o))$ such that
$$
\begin{cases}
0\leq\eta_1\leq 1,\quad \eta_1\equiv 1\text{ in }  B_{{3R}/{4}}(o);\\
|\nabla\eta_1|\leq\frac{C(n)}{R},
\end{cases}
$$
and let $\eta = \eta_1^{\theta+1}$. Direct calculation shows that
\begin{align*}
R^2|\nabla\eta|^2\leq C^2(n)\left( \theta+1\right )^2\eta ^{\frac{2\theta}{\theta+1}}.
\end{align*}

By H\"older inequality and Young inequality, we have
\begin{align}\label{2.39}
\begin{split}
\gamma(\alpha,p,t)\int_{B(o,R)}f^{\theta}|\nabla\eta|^2\leq &\frac{C^2(n)\left(\theta+1\right )^2\gamma(\alpha,p,t)}{R^2} \int_{B(o,R)}f^{\theta}
\eta^{\frac{2\theta}{\theta+1}}\\
\leq &\frac{C^2(n)\left( \theta+1\right )^2\gamma(\alpha,p,t)}{R^2}  \left(\int_{B(o,R)}f^{\theta+1 }\eta^2\right)^{\frac{\theta}{\theta+1}}V^{\frac{ 1}{\theta+1}}\\
\leq &\frac{\beta_{n,p,q,\alpha }}{2}\left[\int_{B(o,R)}f^{ \theta+1}\eta^2 + \left(\frac{C^2(n)\left(\theta+1\right )^2\gamma(\alpha,p,t) }{\beta_{n,p,q,\alpha}R^2}\right)^{ \theta+1 }V\right].
\end{split}
\end{align}
Since $\eta\equiv1$ in $B_{3R/4}$, there holds
\begin{equation}\label{5.8}
\left\|f^{\theta}\right\|_{L^{\chi}(B_{3R/4}(o))}\leq \left\|f^{\theta}\eta^2\right\|_{L^{\chi}(B_R(o))}.
\end{equation}		
Substituting \eqref{2.39} and \eqref{5.8} into \eqref{924} and keeping in mind
$$\left(\frac{a_2t}{\theta^2}\mathbb{S}_{\chi}(M)- \left\|\ric_-\right\|_{L^{\frac{\chi}{\chi-1}}}\right)>0,$$
we arrive at
\begin{equation}
\left\|f^{\theta}\right\|_{L^{\chi}(B_{3R/4}(o))}\leq \frac{\beta_{n,p,q,\alpha}}{\beta(\alpha,p,t)}\left(\frac{C^2(n)\left( \theta+1\right )^2\gamma(\alpha,p,t) }{\beta_{n,p,q,\alpha}R^2}\right)^{ \theta+1 }V.\end{equation}	
Finally, we conclude that
$$\|f \|_{L^{\theta\chi}(B_{3R/4}(o))}\leq \frac{\left(C^2(n)\left( \theta+1\right )^2\gamma(\alpha,p,t)\right)^{\frac{\theta+1}{\theta}}}{\beta_{n,p,q,\alpha}\beta(\alpha,p,t)^{\frac{1}{\theta}}}\left(\frac{V}{R^{2(\theta+1)}}\right)^{\frac{1}{\theta}}.$$	\end{proof}

\begin{proof}[Proof of \thmref{main}:] Firstly, fix some $\alpha>0$ such that $\alpha$ meets the conditions in \lemref{lem31}. Then we can choose a large $t$ such that
$$2\left(\alpha+t+\frac{p}{2}\right)>\frac{2\chi}{\chi-1}\quad\mbox{as}\,\, \chi\in\left(1,\, \frac{n}{n-2}\right)$$
or 
$$2\left(\alpha+t+\frac{p}{2}\right)>\beta^*\geq n \quad\mbox{as}\,\, \chi=\frac{n}{n-2},$$
and the condition \eqref{tc} holds. Once these have been done, we see that there exists a positive constant $C(n,p,q,\beta^*)$ depending on $n,p,q$ and $\beta^*$ such that, if
$$\|\ric_-\|_{L^{\frac{\chi}{\chi-1}}} < C(n,p,q,\beta^*)\mathbb{S}_{\chi}(M),$$
then relation \eqref{lpbpund} holds. Now, by letting $R$ tends to infinity, we conclude that $f=0$, i.e., $\nabla u=0$. Thus, $v$ is a positive constant. This contradicts to $v$ is a solution. We complete the proof.
\end{proof}

\subsection{Local gradient estimate when $\ric_{-} \in L^{\gamma}$ for some $\gamma>\chi/(\chi-1)$.}
\begin{thm}\label{main1}
Let $(M,g)$ be a complete noncompact Riemannian manifold on which the {\it $\chi$-type Sobolev inequality} holds.  Denote $\Lambda=\|\ric_-\|_{L^{\gamma}(B_1)}$ for some $\gamma>\chi/(\chi-1)$. Then, for any $r\leq1$ when
$$a>0\quad  \&\quad  q<\frac{n+3}{n-1}(p-1)\quad \text{or} \quad a<0 \quad \& \quad  q>p-1,
$$
the following local gradient estimate holds for positive solution $v$ to \eqref{equ0},
$$\sup_{ B_{r/2}}\frac{|\nabla v|^2}{v^2}\leq a_5\left(\frac{V}{r^{2\left(\frac{\chi}{\chi-1}+\theta\right)}}\right)^{\frac{1}{\theta}},$$
where $V$ is the volume of $B_r$.
\end{thm}
\begin{proof}
Notice that, by \lemref{41}, we already have
\begin{equation}\label{1015}
\frac{\beta_{n,p,q,\alpha}}{2} \int_{B_r} f^{\theta+1}\eta^2 + \frac{a_2t}{\theta^2}\mathbb{S}_{\chi}(M) \left\|f^{\theta} \eta^2\right\|_{L^{\chi}}
\leq \int_{B_r} \ric_-f^{\theta}\eta^2+\gamma(\alpha,p,t)\int_{B_r} f^{\theta}|\nabla\eta|^2.
\end{equation}	
By Holder's inequality, we have
\begin{equation}
\int_\Omega \ric_-f^{\theta}\eta^2\leq \Lambda\|f^{\frac{\theta}{2}}\eta\|_{L^{\frac{2\gamma}{\gamma-1}}}^2.
\end{equation}
Notice that
$$\frac{2\gamma}{\gamma-1}\in(2,\, 2\chi).$$
By interpolation inequality, we have
$$\left\|f^{\frac{\theta}{2}}\eta\right\|_{L^{\frac{2\gamma}{\gamma-1}}}\leq \varepsilon\left\|f^{\frac{\theta}{2}}\eta\right\|_{L^{2\chi}}+\varepsilon^{-\frac{\chi}{(\chi-1)\gamma-\chi}}\left\|f^{\frac{\theta}{2}}\eta\right\|_{L^2}.$$
Hence
\begin{equation*}
\int_\Omega \ric_-f^{\theta}\eta^2\leq 2\Lambda\varepsilon^2\left\|f^{\frac{\theta}{2}}\eta\right\|_{L^{2\chi}}^2+2\Lambda\varepsilon^{-\frac{2\chi}{(\chi-1)\gamma-\chi}}\left\|f^{\frac{\theta}{2}}\eta\right\|_{L^2}^2.\end{equation*}
Without loss of generality, we assume $\Lambda>0$. Now, Let
$$\varepsilon=\varepsilon(\alpha,p,t)=\frac{1}{2\theta}\left(\frac{a_2t\mathbb{S}_{\chi}(M)}{\Lambda}\right)^{\frac{1}{2}}.$$
Substituting the above into \eqref{1015}, we arrive at
\begin{align}\label{1016}
&\frac{\beta_{n,p,q,\alpha}}{2} \int_{B_r} f^{\theta+1}\eta^2 + \frac{a_2t}{2\theta^2}\mathbb{S}_{\chi}(M) \left\|f^{\theta} \eta^2\right\|_{L^{\chi}}\\ \nonumber
\leq\,&\gamma(\alpha,p,t)\int_{B_r} f^{\theta}|\nabla\eta|^2+2\Lambda\varepsilon^{-\frac{2\chi}{(\chi-1)\gamma-\chi}}\int_{B_r}f^{\theta}\eta^2.
\end{align}
By almost the same arguments as the proof of \lemref{lpbound}, we conclude that  when $r\leq1$, there exists $a_3 = a_3(n,\alpha,p,q,t,\Lambda)>0$ such that
\begin{align}\label{lpbound1}
\|f \|_{L^{\theta\chi}(B_{3r/4}(o))}\leq a_3\left(\frac{V}{r^{2(\theta+1)}}\right)^{\frac{1}{\theta}},
\end{align}
where $V$ denotes the volume of geodesic ball $B_r(o)$.

Now, we fix $\alpha=\alpha_0$ such that $\alpha_0$ satisfies the conditions in \lemref{lem31}. Once $\alpha$ has been fixed, $\theta$ can be regarded as a function with respect to $t$. That is to say, when $\alpha=\alpha_0$ is fixed,
$$\theta=\theta(t)=\alpha_0+t+\frac{p}{2}-1.$$
On the other hand,  when $r\leq1$, the test function $\eta$ constructed below satisfies
$$1=\|\eta\|_{L^{\infty}}\leq\|\nabla \eta\|_{L^{\infty}}.$$
Hence, \eqref{1016} can be rewritten as
\begin{equation}\label{1016*}
\frac{\beta_{n,p,q,\alpha}}{2} \int_\Omega f^{\theta+1}\eta^2 + \frac{a_2t}{2\theta^2}\mathbb{S}_{\chi}(M) \left\|f^{\theta} \eta^2\right\|_{L^{\chi}}
\leq\left(\gamma(\alpha,p,t)+2\Lambda\varepsilon^{-\frac{2\chi}{(\chi-1)\gamma-\chi}}\right)|\nabla\eta|_{L^{\infty}}^2\int_\Omega f^{\theta}.
\end{equation}

Now, denote by
$$\zeta(t)=\frac{2\theta^2\left(\gamma(\alpha_0,p,t)+2\Lambda\varepsilon^{-\frac{2\chi}{(\chi-1)\gamma-\chi}}\right)}{a_2t\mathcal{S}_\chi(M)}.$$
It follows from \eqref{1016*} that
\begin{equation}\label{931}
\left\|f^{\theta}\eta^2\right\|_{L^{\chi}}\leq \zeta(t)|\nabla\eta|_{L^{\infty}}^2\int_\Omega f^{\theta}.
\end{equation}
Let
$$\Omega_k = B_{r_k}(o) \quad \text{with} \quad r_k = \frac{r}{2}+\frac{r}{4^k},$$
 and choose $\eta_k\in C^{\infty}(\Omega_k )$ satisfying
\begin{align*}
\begin{cases}
0\leq \eta_k\leq1, \quad \eta_k\equiv 1\text{ in }B_{r_{k+1}}(o);\\
|\nabla\eta_k|\leq \frac{C4^k}{r}.
\end{cases}
\end{align*}

Now, for any $t_0$ satisfying the condition in \lemref{41}, we denote
$$\theta_0=\alpha_0+t_0+\frac{p}{2}-1.$$
Moreover, we let $\beta_k=\theta_0\chi^{k}$ and $t=t_k$ such that
$$t_k+\frac{p}{2}+\alpha_0-1=\beta_k.$$
By substituting $\theta=\theta_k$ and $ \eta$ by $\eta_k$ in \eqref{931}, we arrive at
$$\|f\|_{L^{\beta_{k+1}}(\Omega_{k+1})}\leq \zeta(t_k)^{\frac{1}{\beta_k}} 16 ^{\frac{k}{\beta_k}}r^{-\frac{2}{\beta_k}}\|f\|_{L^{\beta_k}(\Omega_k)}.$$
Hence,
\begin{equation}\label{1929}
\|f\|_{L^{\beta_{k+1}}(\Omega_{k+1})}
\leq
\prod_{i=1}^k\zeta(t_i)^{\frac{1}{\beta_i}} 16 ^{\sum_{i=1}^k\frac{i}{\beta_i}}r^{-\sum_{i=1}^k\frac{2}{\beta_i}}\|f\|_{L^{\beta_1}(\Omega_k)}.
\end{equation}
Notice that
$$\zeta(t)\leq c(p,\alpha_0,\Lambda, \mathbb{S}_{\chi}(M))t^{\frac{(\chi-1)\gamma}{(\chi-1)\gamma-\chi}}.$$
Straightforward calculation shows that
$$
\prod_{i=1}^\infty\zeta(t_i)^{\frac{1}{\beta_i}}<\infty,\quad\quad\sum_{k=1}^{\infty}\frac{1}{\beta_k}  =\frac{1}{\theta_0(\chi-1)},\quad\quad \sum_{k=1}^{\infty}\frac{k}{\beta_k} <\infty.
$$
By letting $k\rightarrow\infty$ in \eqref{1929}, we arrive at
\begin{align*}
 \|f\|_{L^{\infty}(B_{r/2}(o))}
\leq &
a_{4}r^{-\frac{2}{\theta_0(\chi-1)}} \|f\|_{L^{\beta}(B_{3r/4}(o))}.
\end{align*}
By substituting \eqref{lpbound1} into the above, we finally arrive at
\begin{equation*}\label{104}
\|f\|_{L^{\infty}(B_{r/2}(o))}
\leq
a_5(\alpha_0,t_0, p, q, \Lambda, \mathbb{S}_\chi(M))\left(\frac{V}{r^{2\left(\frac{\chi}{\chi-1}+\theta_0\right)}}\right)^{\frac{1}{\theta_0}}.
\end{equation*}
Thus, we complete the proof.
\end{proof}

\begin{proof}[Proof of \thmref{main3}:]
Since $\chi\leq n/(n-2)$, it is easy to see that $\chi/(\chi-1)\geq n/2$. Hence $\gamma>n/2$.
By the relative volume comparison theorem under the integral bounded Ricci curvature due to Peterson and Wei \cite{PW}, there holds
\begin{equation}
\vol(B_r)\leq \omega_nr^n+C(n,\gamma)\|\ric_-\|^{\gamma}_{L^\gamma} r^{2\gamma}.
\end{equation}
Hence, by substituting the above into \thmref{main1}, then letting $r=1$, we conclude the conclusion.
\end{proof}

\section{Laplace Case: Proof of Theorem \ref{theorem0}}
In the previous section (see \thmref{main}), we have shown that the conclusion of Theorem \ref{theorem0} holds for $q<\frac{n+3}{n-1}$.  In this section, we shall prove the remaining case of
$$q\in\left[\frac{n+3}{n-1},~\frac{n+2}{(n-2)_+}\right).$$

Throughout this section, we assume $v$ be a positive solution of \eqref{se} on $B(o,R)\subset M$. First, we recall some auxiliary functions and point-wise estimates developed by Lu in \cite{Lu}.  As categorized in \cite{Lu}, we need to consider dimensions greater than or equal to $4$ and less than $4$ respectively.

\subsection{\textbf{Case 1: $n\ge4 \quad\&\quad q\in\left[\frac{n+3}{n-1},~\frac{n+2}{n-2}\right)$\ }}\

\subsubsection{\textbf{First auxiliary function $F$ and  estimation of the leading coefficients. }}\

 For $\theta\neq0$, let $\omega=v^{-\theta}$.
A straightword calculation shows that
\begin{align}\label{2.2*}
    \Delta\omega=\left(1+\frac{1}{\theta}\right)\frac{|\nabla\omega|^2}{\omega}+\theta\omega v^{q-1}.
\end{align}
For undetermined real numbers $\varepsilon>0$ and $d>0$, define the first type auxiliary function:
\begin{align}\label{2.3*}
    F=(v+\varepsilon)^{-\theta}\left(\frac{|\nabla\omega|^2}{\omega^2}+dv^{q-1}\right).
\end{align}
\begin{lem}\label{lemma2.1*}
$(\mathrm{Lemma\,2.1}$ and Lemma 6.1 in \cite{Lu}$)$ There holds:
\begin{align}\label{2.4*}
\begin{split}
(v+\varepsilon)^\theta\Delta F=\,&2\omega^{-2}\left|\nabla^2\omega-\frac{\Delta \omega}{n}g\right|^2+2\omega^{-2}\mathrm{Ric}(\nabla\omega,\nabla\omega)\\
&+2\left(\frac{1}{\theta}-\frac{\varepsilon}{v+\varepsilon}\right)(v+\varepsilon)^\theta\langle\nabla F,\nabla\ln\omega\rangle\\
&+U\frac{|\nabla\omega|^4}{\omega^4}+V\frac{|\nabla\omega|^2}{\omega^2}v^{q-1}+Wv^{2(q-1)},
\end{split}
\end{align}
where
\begin{align*}
U=&\frac{2}{n}\left(1+\frac{1}{\theta}\right)^2+\left(\frac{1}{\theta}-1\right)\frac{v^2}{(v+\varepsilon)^2}+2\left(1-\frac{1}{\theta}\right)\frac{v}{v+\varepsilon}-2,\\
V=&\frac{4}{n}(1+\theta)+2(1-q)+\frac{d(q-1)}{\theta^2}(q-2\theta)+\frac{v}{v+\varepsilon}\left\{\theta-d\left(\frac{1}{\theta}-1\right)\left(1+\frac{\varepsilon}{v+\varepsilon}\right)\right\},\\
W=&\frac{2\theta^2}{n}+d\left(\frac{\theta v}{v+\varepsilon}+1-q\right).
\end{align*}
Moreover, for  $n\ge4$ and $q\in\left[\frac{n+3}{n-1},~\frac{n+2}{n-2}\right)$, there exist $\theta=\theta(n,q)\in\left(0,\frac{2}{n-2}\right)$, $d=d(n,q)>0$, $L=L(n,q)>0$ and $\widetilde{M}=\widetilde{M}(n,q)>0$ such that for any $\varepsilon>0$,
\begin{align*}
U&\ge U_0>0\\
V&\ge V_0-\widetilde{M}\bar\chi_{\left\{x\in B(o,R):\,v(x)<L\varepsilon\right\}},\\
W&\ge W_0-\widetilde{M}\bar\chi_{\left\{x\in B(o,R):\,v(x)<L\varepsilon\right\}},
\end{align*}
where $U_0$, $V_0$, $W_0$ are positive constants depending only on $n$ and $q$, and $\bar\chi$ denotes the characteristic function.
\end{lem}

Next, we shall provide a precise estimate for $\Delta F$.
\begin{lem}\label{lemma2.3*}
Let $n\ge4$ and $q\in\left[\frac{n+3}{n-1},~\frac{n+2}{n-2}\right)$. Then there exist $\theta=\theta(n,q)\in\left(0,\frac{2}{n-2}\right)$, $d=d(n,q)>0$, $C_0=C_0(n,q)>0$ and $M_0=M_0(n,q)>0$,  such that the following holds point-wisely in $B(o,R)$:
\begin{align}\label{2.5*}
\Delta F\ge-2\mathrm{Ric}_-F-\frac{2}{\theta}|\nabla F||\nabla\ln\omega|+C_0(v+\varepsilon)^\theta F^2-M_0\varepsilon^{q-1}F.
\end{align}
\end{lem}

\begin{proof}
By Lemma \ref{lemma2.1*}, we know that there exist constants
$$\theta=\theta(n,q)\in\left(0,\frac{2}{n-2}\right),\quad d=d(n,q)>0,\quad L=L(n,q)>0\quad \mbox{and}\quad \widetilde{M}=\widetilde{M}(n,q)>0$$
such that for any $\varepsilon>0$,
\begin{align}\label{2.6*}
\begin{split}
(v+\varepsilon)^\theta\Delta F\ge\,&2\omega^{-2}\mathrm{Ric}(\nabla\omega,\nabla\omega) +2\left(\frac{1}{\theta}-\frac{\varepsilon}{v+\varepsilon}\right)(v+\varepsilon)^\theta\langle\nabla F,\nabla\ln\omega\rangle\\
&+U_0\frac{|\nabla\omega|^4}{\omega^4}+V_0\frac{|\nabla\omega|^2}{\omega^2}v^{q-1}+W_0v^{2(q-1)}\\
&-\widetilde{M}v^{q-1}\left(\frac{|\nabla\omega|^2}{\omega^2}+v^{q-1}\right)\bar{\chi}_{\left\{x\in B(o,R): v(x)<L\varepsilon\right\}},
\end{split}
\end{align}
where $U_0$, $V_0$, $W_0$ are positive constants that depend only on $n$ and $q$.  Notice that
\begin{align}\label{2.7*}
2\left(\frac{1}{\theta}-1+\frac{v}{v+\varepsilon}\right)(v+\varepsilon)^\theta\langle\nabla F,\nabla\ln\omega\rangle \ge-\frac{2}{\theta}(v+\varepsilon)^\theta|\nabla F||\nabla\ln\omega|,
\end{align}

\begin{align}\label{2.8*}
\begin{split}
U_0\frac{|\nabla\omega|^4}{\omega^4}+V_0\frac{|\nabla\omega|^2}{\omega^2}v^{q-1}+ W_0v^{2(q-1)}\ge&\min\left\{U_0,\frac{V_0}{2},
W_0\right\}\left(\frac{|\nabla\omega|^2}{\omega^2}+v^{q-1}\right)^2\\
\ge&C_0(n,q)\left(\frac{|\nabla\omega|^2}{\omega^2}+dv^{q-1}\right)^2\\
=&C_0(n,q)(v+\varepsilon)^{2\theta}F^2,
\end{split}
\end{align}
and
\begin{align}\label{2.9*}
\begin{split}
-\widetilde{M}v^{q-1}\left(\frac{|\nabla\omega|^2}{\omega^2}+v^{q-1}\right)\chi_{\left\{x\in B(o,R):v(x)<L\varepsilon\right\}}\ge &-\widetilde{M}(L\varepsilon)^{q-1}\left(\frac{|\nabla\omega|^2}{\omega^2}+v^{q-1}\right)\\
\ge&-M_0(n,q)\varepsilon^{q-1}\left(\frac{|\nabla\omega|^2}{\omega^2}+dv^{q-1}\right)\\
=&-M_0(n,q)\varepsilon^{q-1}(v+\varepsilon)^\theta F,
\end{split}
\end{align}
where $C_0(n,q)$ and $M_0(n,q)$ are positive numbers and depend only on $n$ and $q$.

Now, substituting \eqref{2.7*}, \eqref{2.8*} and \eqref{2.9*} into \eqref{2.6*} and dividing the both sides by $(v+\varepsilon)^\theta$ yields
\begin{align}\label{2.10*}
\Delta F\ge \,&2(v+\varepsilon)^{-\theta}\omega^{-2}\mathrm{Ric}(\nabla\omega,\nabla\omega)-2\left(\frac{1}{\theta}+1\right)|\nabla F||\nabla\ln\omega| \nonumber\\
&+C_0(v+\varepsilon)^\theta F^2-M_0\varepsilon^{q-1}F.
\end{align}
On the other hand, from \eqref{2.3*} we obtain
\begin{align*}
2(v+\varepsilon)^{-\theta}\omega^{-2}\mathrm{Ric}(\nabla\omega,\nabla\omega)\ge&-2(v+\varepsilon)^{-\theta}\mathrm{Ric_-}\frac{|\nabla\omega|^2}{\omega^2}\ge-2\mathrm{Ric_-}F.
\end{align*}
Substituting the above inequality into \eqref{2.10*}, we finish the proof of Lemma \ref{lemma2.3*}.
\end{proof}

\subsubsection{\textbf{Integral estimate}}\

Now, we are going to establish a key integral inequality of $F$.
\begin{lem}\label{lemma2.4*}
Let $(M,g)$ be a complete manifold on which the $\chi$-type Sobolev inequality holds. Let $n\ge4$, $q\in\left[\frac{n+3}{n-1},~\frac{n+2}{n-2}\right)$ and $\Omega=B(o,R)$. Define $\theta$, $d$, $C_0$ and $M_0$ as in Lemma \ref{lemma2.3*}. Then, for
\begin{align}\label{2.11*}
t\in\left(\max\left\{\frac{8}{C_0\theta^2},1\right\},~+\infty\right),
\end{align}
the following holds
\begin{align}\label{2.12*}
\begin{split}
&C_0\varepsilon^\theta\int_\Omega F^{t+2}\eta^2+\frac{\mathbb{S}_\chi(M)}{2t}\left\|F^{t+1}\eta^2\right\|_{L^\chi(\Omega)}\\
\le&4\int_\Omega\mathrm{Ric}_-F^{t+1}\eta^2+2M_0\varepsilon^{q-1}\int_\Omega F^{t+1}\eta^2+\frac{12}{t}\int_\Omega F^{t+1}|\nabla\eta|^2,
\end{split}
\end{align}
where $\eta\ge0$ and $\eta\in C_0^\infty(\Omega)$.
\end{lem}

\begin{proof}
Let $\eta\in C_0^\infty(\Omega)$ be a nonnegative function. By multiplying $F^t\eta^2$ on the both sides of \eqref{2.5*} ($t>1$ will be determined later) and integration by parts, we arrive at
\begin{align}\label{2.13*}
\begin{split}
&2\int_\Omega\mathrm{Ric}_-F^{t+1}\eta^2+\frac{2}{\theta}\int_\Omega F^t|\nabla F||\nabla\ln\omega|\eta^2+M_0\varepsilon^{q-1}\int_\Omega F^{t+1}\eta^2-2\int_\Omega F^t\langle\nabla F,\nabla\eta\rangle\eta\\
\ge&t\int_\Omega F^{t-1}|\nabla F|^2\eta^2+C_0\int_\Omega(v+\varepsilon)^\theta F^{t+2}\eta^2.
\end{split}
\end{align}
Notice that
\begin{align*}
\frac{2}{\theta}\int_\Omega F^t|\nabla F||\nabla\ln\omega|\eta^2\le&\frac{t}{4}\int_\Omega F^{t-1}|\nabla F|^2\eta^2+\frac{4}{t\theta^2}\int_\Omega F^{t+1}|\nabla\ln\omega|^2\eta^2\\
\le&\frac{t}{4}\int_\Omega F^{t-1}|\nabla F|^2\eta^2+\frac{4}{t\theta^2}\int_\Omega(v+\varepsilon)^\theta  F^{t+2}\eta^2
\end{align*}
and
\begin{align*}
-2\int_\Omega F^t\langle\nabla F,\nabla\eta\rangle\eta\le\frac{t}{4}\int_\Omega F^{t-1}|\nabla F|^2\eta^2+\frac{4}{t}\int_\Omega F^{t+1}|\nabla\eta|^2.
\end{align*}
Substituting the above two inequalities into \eqref{2.13*}, we conclude
\begin{align*}
&\frac{t}{2}\int_\Omega F^{t-1}|\nabla F|^2\eta^2+\left(C_0-\frac{4}{t\theta^2}\right)\int_\Omega(v+\varepsilon)^\theta  F^{t+2}\eta^2\\
\le&2\int_\Omega\mathrm{Ric}_-F^{t+1}\eta^2+M_0\varepsilon^{q-1}\int_\Omega F^{t+1}\eta^2+\frac{4}{t}\int_\Omega F^{t+1}|\nabla\eta|^2.
\end{align*}
Let
\begin{align*}
t\in\left(\max\left\{\frac{8}{C_0\theta^2},1\right\},~+\infty\right),
\end{align*}
then we have
\begin{align}\label{2.14*}
\begin{split}
&t\int_\Omega F^{t-1}|\nabla F|^2\eta^2+C_0\int_\Omega(v+\varepsilon)^\theta  F^{t+2}\eta^2\\
\le&4\int_\Omega\mathrm{Ric}_-F^{t+1}\eta^2+2M_0\varepsilon^{q-1}\int_\Omega F^{t+1}\eta^2+\frac{8}{t}\int_\Omega F^{t+1}|\nabla\eta|^2.
\end{split}
\end{align}
By $\chi$-type Sobolev inequality, there holds
\begin{align*}
\mathbb{S}_\chi(M)\left\|F^{\frac{t+1}{2}}\eta\right\|_{L^{2\chi}(\Omega)}^2\le\int_\Omega\left|\nabla\left(F^{\frac{t+1}{2}}\eta\right)\right|^2.
\end{align*}
Hence,
\begin{align*}
\mathbb{S}_\chi(M)\left\|F^{t+1}\eta^2\right\|_{L^{\chi}(\Omega)}^2\le\frac{(t+1)^2}{2}\int_\Omega F^{t-1}|\nabla F|^2\eta^2+2\int_\Omega F^{t+1}|\nabla\eta|^2.
\end{align*}
Substituting the above inequality into \eqref{2.14*}, then we obtain
\begin{align*}
&\mathbb{S}_\chi(M)\frac{2t}{(t+1)^2}\left\|F^{t+1}\eta^2\right\|_{L^{\chi}(\Omega)}^2+C_0\int_\Omega(v+\varepsilon)^\theta  F^{t+2}\eta^2\\
\le&4\int_\Omega\mathrm{Ric}_-F^{t+1}\eta^2+2M_0\varepsilon^{q-1}\int_\Omega F^{t+1}\eta^2+\left[\frac{8}{t}+\frac{4t}{(t+1)^2}\right]\int_\Omega F^{t+1}|\nabla\eta|^2.
\end{align*}
Therefore, we obtain
\begin{align*}
&\frac{\mathbb{S}_\chi(M)}{2t}\left\|F^{t+1}\eta^2\right\|_{L^{\chi}(\Omega)}^2+C_0\varepsilon^\theta\int_\Omega F^{t+2}\eta^2\\
\le&4\int_\Omega\mathrm{Ric}_-F^{t+1}\eta^2+2M_0\varepsilon^{q-1}\int_\Omega F^{t+1}\eta^2+\frac{12}{t}\int_\Omega F^{t+1}|\nabla\eta|^2.
\end{align*}
Combining above, we finish the proof of Lemma \ref{lemma2.4*}.

\end{proof}

Next, we shall provide the following $L^{(t+1)\chi}$ bound of $(v+1)^{-\theta}v^{q-1}$.
\begin{lem}\label{lemma2.5*}
Let $(M,g)$ be a complete manifold on which the $\chi$-type Sobolev inequality holds. Assume $v$ is a positive solution to \eqref{se} on the geodesic ball $B(o,R)\subset M$. Furthermore, let $n\ge4$, $q\in\left[\frac{n+3}{n-1},~\frac{n+2}{n-2}\right)$ and $\Omega=B(o,R)$. Define $t$ and $\theta$ as in Lemma \ref{lemma2.4*}. Assume further that
\begin{align*}
H(t)=\frac{\mathbb{S}_\chi (M)}{2t}-4\left\|\mathrm{Ric_-}\right\|_{L^{\frac{\chi}{\chi-1}}}>0.
\end{align*}
Then there exists $C_3=C_3(n,q,t)>0$ such that
\begin{align}\label{2.15*}
H(t)\left\{\int_{B\left(o,\frac{3}{4}R\right)}\left[(v+1)^{-\theta}v^{q-1}\right]^{(t+1)\chi}\right\}^\frac{1}{\chi}\le C_3\left(\frac{1}{R^{2(t+2)}\varepsilon^{\theta(t+1)}}+\varepsilon^{(q-1)+(q-\theta-1)(t+1)}\right)V,
\end{align}
where $V$ denotes the volume of geodesic ball $B(o,R)$ and $\varepsilon\in(0,1)$ is any positive number.
\end{lem}

\begin{proof}
By H\"older's inequality, we have
\begin{align*}
\int_\Omega\mathrm{Ric}_-F^{t+1}\eta^2\le\left\|\mathrm{Ric_-}\right\|_{L^{\frac{\chi}{\chi-1}}}\left\|F^{t+1}\eta^2\right\|_{L^{\chi}(\Omega)}.
\end{align*}
Substituting this into \eqref{2.12*}, we conclude that
\begin{align}\label{2.16*}
C_0\varepsilon^\theta\int_\Omega F^{t+2}\eta^2+H(t)\left\|F^{t+1}\eta^2\right\|_{L^\chi(\Omega)}\le2M_0\varepsilon^{q-1}\int_\Omega F^{t+1}\eta^2+\frac{12}{t}\int_\Omega F^{t+1}|\nabla\eta|^2.
\end{align}
Now, choose $\eta_1\in C_0^\infty(\Omega)$ such that
\begin{align*}
\begin{cases}
0\le\eta_1\le1,\quad\eta_1\equiv 1~in~B\left(o,\frac{3}{4}R\right);\\[3mm]
|\nabla\eta_1|\le\frac{C(n)}{R}.
\end{cases}
\end{align*}
and let $\eta=\eta_1^{t+2}$. Direct calculation shows that
\begin{align*}
|\nabla\eta|^2\le\frac{C^2(n)}{R^2}(t+2)^2\eta^{\frac{2(t+1)}{t+2}}.
\end{align*}
By Young inequality, we obtain
\begin{align}\label{2.17*}
\begin{split}
\frac{12}{t}\int_\Omega F^{t+1}|\nabla\eta|^2\le&\frac{C^2(n)}{R^2}\frac{12(t+2)^2}{t}\int_\Omega F^{t+1}\eta^{\frac{2(t+1)}{t+2}}\\
\le&\frac{C_0\varepsilon^\theta}{2}\int_\Omega F^{t+2}\eta^2+C_1(n,q,t)\varepsilon^{-\theta(t+1)}\frac{V}{R^{2(t+2)}},
\end{split}
\end{align}
where $C_1(n,q,t)$ is positive and depends only on $n$, $q$ and $t$.

Set
\begin{align*}
\widetilde{\Omega}=\left\{x\in\Omega:F\ge\frac{4M_0}{C_0}\varepsilon^{q-\theta-1}\right\},
\end{align*}
then we have
\begin{align}\label{2.18*}
\begin{split}
2M_0\varepsilon^{q-1}\int_\Omega F^{t+1}\eta^2=&2M_0\varepsilon^{q-1}\int_{\widetilde{\Omega}} F^{t+1}\eta^2+2M_0\varepsilon^{q-1}\int_{\Omega\setminus\widetilde{\Omega}}F^{t+1}\eta^2\\
\le&\frac{C_0\varepsilon^\theta}{2}\int_{\widetilde{\Omega}} F^{t+2}\eta^2+2M_0\varepsilon^{q-1}\int_{\Omega\setminus\widetilde{\Omega}}\left(\frac{4M_0}{C_0}\varepsilon^{q-\theta-1}\right)^{t+1}\\
\le&\frac{C_0\varepsilon^\theta}{2}\int_\Omega F^{t+2}\eta^2+C_2(n,q,t)\varepsilon^{(q-1)+(q-\theta-1)(t+1)}V,
\end{split}
\end{align}
where $C_2(n,q,t)$ is a positive and depends only on $n$, $q$ and $t$.

Substituting \eqref{2.17*} and \eqref{2.18*} into \eqref{2.16*} yields
\begin{align}\label{2.19*}
\begin{split}
H(t)\left\|F^{t+1}\eta^2\right\|_{L^\chi(\Omega)}\le C_1(n,q,t)\varepsilon^{-\theta(t+1)}\frac{V}{R^{2(t+2)}}+C_2(n,q,t)\varepsilon^{(q-1)+(q-\theta-1)(t+1)}V.
\end{split}
\end{align}
From the definition of $F$,
we conclude that
\begin{align*}
&H(t)\left\{\int_{B\left(o,\frac{3}{4}R\right)}\left[(v+\varepsilon)^{-\theta}\left(\frac{|\nabla\omega|^2}{\omega^2}+dv^{q-1}\right)\right]^{(t+1)\chi}\right\}^\frac{1}{\chi}\\
\le&C_1(n,q,t)\varepsilon^{-\theta(t+1)}\frac{V}{R^{2(t+2)}}+C_2(n,q,t)\varepsilon^{(q-1)+(q-\theta-1)(t+1)}V.
\end{align*}
Hence,
\begin{align*}
H(t)\left\{\int_{B\left(o,\frac{3}{4}R\right)}\left[(v+\varepsilon)^{-\theta}v^{q-1}\right]^{(t+1)\chi}\right\}^\frac{1}{\chi}\le C_3(n,p,t)\left(\frac{1}{R^{2(t+2)}\varepsilon^{\theta(t+1)}}+\varepsilon^{(q-1)+(q-\theta-1)(t+1)}\right)V,
\end{align*}
where $C_3(n,q,t)$ is a positive number and depends only on $n$, $q$ and $t$. We finish the proof of Lemma \ref{lemma2.5*}.
\end{proof}

Now, we are ready to prove Theorem \ref{theorem0} in the case $n\ge 4$ and $\frac{n+3}{n-1}\le q<\frac{n+2}{n-2}$.

\begin{proof}[\bf Proof of Theorem \ref{theorem0} for the case $n\ge 4$ $\&$ $\frac{n+3}{n-1}\le q<\frac{n+2}{n-2}$:]
Since $\theta\in\left(0,\,\frac{2}{n-2}\right)$, we have
\begin{align*}
q-\theta-1>0.
\end{align*}
Next, we choose a large $t$ such that, as $\chi\in(1, \frac{n}{n-2})$,
\begin{align}\label{2.20*}
t+3-\frac{2\chi}{\chi-1}>0\quad\mbox{and}\quad \frac{1}{\theta}\left[(q-1)+(q-\theta-1)(t+1)\right]-\frac{2\chi}{\chi-1}>0,
\end{align}
or as $\chi= \frac{n}{n-2}$,
\begin{align}\label{2.20*}
t+3-\beta^*>0\quad\mbox{and}\quad \frac{1}{\theta}\left[(q-1)+(q-\theta-1)(t+1)\right]-\beta^*>0,
\end{align}
and the condition \eqref{2.11*} holds. Once these have been done, we see that there exists a positive constant $C(n,q,\beta^*)$ depending on $n$, $q$ and $\beta^*$ such that, if
\begin{align*}
\|\mathrm{Ric}_-\|_{L^\frac{\chi}{\chi-1}}\le C(n,q,\beta^*)\mathbb{S}_\chi(M),
\end{align*}
then Lemma \ref{lemma2.5*} holds.

By the fact that $\varepsilon\in(0,1)$, we infer from Lemma \ref{lemma2.5*} that
\begin{align*}
 H(t)\left\{\int_{B\left(o,\frac{3}{4}R\right)}\left[(v+1)^{-\theta}v^{q-1}\right]^{(t+1)\chi}\right\}^\frac{1}{\chi}\le C_3\left(\frac{1}{R^{2(t+2)}\varepsilon^{\theta(t+1)}}+\varepsilon^{(q-1)+(q-\theta-1)(t+1)}\right)V,
 \end{align*}
 where $\varepsilon$, $\theta$, $t$ and $H(t)$ are defined in Lemma \ref{lemma2.5*}.

Let
\begin{align*}
\varepsilon=R^{-\frac{1}{\theta}}\quad(R\ge1),
\end{align*}
then we have
\begin{align*}
H(t)\left\{\int_{B\left(o,\frac{3}{4}R\right)}\left[(v+1)^{-\theta}v^{q-1}\right]^{(t+1)\chi}\right\}^\frac{1}{\chi}\le C_3\left(\frac{1}{R^{t+3}}+\frac{1}{R^{\frac{1}{\theta}\left[(q-1)+(q-\theta-1)(t+1)\right]}}\right)V.
\end{align*}
Since 
$$\mathrm{vol}\left(B(o,R)\right)=O(R^{\frac{2\chi}{\chi-1}})\quad \mbox{or}\quad \mathrm{vol}\left(B(o,R)\right)=O(R^{\beta^*}),$$ 
by setting $\beta^*_\chi= \frac{2\chi}{\chi-1}$ as $1<\chi<\frac{n}{n-2}$ and $\beta^*_\chi=\beta^*$ as $\chi=\frac{n}{n-2}$, we obtain
\begin{align}\label{2.21*}
H(t)\left\{\int_{B\left(o,\frac{3}{4}R\right)}\left[(v+1)^{-\theta}v^{q-1}\right]^{(t+1)\chi}\right\}^\frac{1}{\chi}\le C_4\left(\frac{1}{R^{t+3-\beta^*_\chi}}+\frac{1}{R^{\frac{1}{\theta}\left[(q-1)+(q-\theta-1)(t+1)\right]-\beta^*_\chi}}\right).
\end{align}
Combining \eqref{2.20*} and \eqref{2.21*} together and letting $R\longrightarrow+\infty$, we arrive at
\begin{align*}
\left\{\int_M\left[(v+1)^{-\theta}v^{q-1}\right]^{(t+1)\chi}\right\}^\frac{1}{\chi}=0.
\end{align*}
Therefore, $v\equiv0$. This contradicts to the fact $v$ is a positive solution. We complete the proof.
\end{proof}

\subsection{\textbf{Case 2: $n\in\{2,3\} \quad\&\quad q\in\left[\frac{n+3}{n-1},~\frac{n+2}{(n-2)_+}\right)$}}\

\subsubsection{\textbf{Second auxiliary function $G$ and  estimation of the leading coefficients. }}\

 For $\theta\neq0$, let $\omega=(v+\varepsilon)^{-\theta}$, then we have,

\begin{align}\label{2.23*}
    \Delta\omega=\left(1+\frac{1}{\theta}\right)\frac{|\nabla\omega|^2}{\omega}+\theta\omega \frac{v^q}{v+\varepsilon}.
\end{align}
For undetermined real numbers $\varepsilon>0$ and $d>0$, define the second type auxiliary function:
\begin{align}\label{2.24*}
    G=(v+\varepsilon)^{-\theta}\left(\frac{|\nabla\omega|^2}{\omega^2}+dv^{q-1}\right).
\end{align}
\begin{lem}[Lemma 2.2 and Lemma 6.5 in \cite{Lu}]\label{lemma2.7*}
There holds:
\begin{align*}
\begin{split}
\omega^{-1}\Delta G=&2\omega^{-2}\left|\nabla^2\omega-\frac{\Delta \omega}{n}g\right|^2+2\omega^{-2}\mathrm{Ric}(\nabla\omega,\nabla\omega)+\frac{2}{\theta}\omega^{-1}\langle\nabla G,\nabla\ln\omega\rangle\\
&+U\frac{|\nabla\omega|^4}{\omega^4}+V\frac{|\nabla\omega|^2}{\omega^2}v^{q-1}+Wv^{2(q-1)},
\end{split}
\end{align*}
where
\begin{align*}
U=&\left[\frac{2}{n}\left(1+\frac{1}{\theta}\right)-1\right]\left(1+\frac{1}{\theta}\right),\\
V=&\left[\frac{4}{n}(1+\theta)+2+\theta\right]\frac{v}{v+\varepsilon}-2q+\frac{2d}{\theta}\left(\frac{q-1}{\theta}\frac{v+\varepsilon}{v}-1\right)\\
&+d\left[\frac{(q-1)(q-2)}{\theta^2}\frac{(v+\varepsilon)^2}{v^2}+1+\frac{1}{\theta}-\frac{2(q-1)}{\theta}\frac{v+\varepsilon}{v}\right],\\
W=&\frac{2\theta^2}{n}\frac{v^2}{(v+\varepsilon)^2}+d\left(\frac{\theta v}{v+\varepsilon}+1-q\right).
\end{align*}

Moreover, for $n\in\{2,3\}$ and $q\in\left[\frac{n+3}{n-1}, ~\frac{n+2}{(n-2)_+}\right)$, there exist constants $\theta=\theta(n,q)\in\left(0,\, q-1\right)$ if $n=2$ or $\theta=\theta(n,q)\in\left(0,\, \min\left\{2,q-1\right\}\right)$ if $n=3$, $d=d(n,q)>0$, $L=L(n,q)>0$ and $\widetilde{M}=\widetilde{M}(n,q)>0$ such that for any $\varepsilon>0$,
\begin{align*}
U&\ge U_0>0,\\
V&\ge V_0-\widetilde{M}\bar{\chi}_{\left\{x\in B(o,R):\, v(x)<L\varepsilon\right\}},\\
W&\ge W_0-\widetilde{M}\bar{\chi}_{\left\{x\in B(o,R):\, v(x)<L\varepsilon\right\}},
\end{align*}
where $U_0$, $V_0$ and $W_0$ are positive constants depending only on $n$ and $q$.
\end{lem}

\begin{proof}[\bf {Proof of Theorem \ref{theorem0} for the case $n\in\{2,3\}$ $\&$ $\frac{n+3}{n-1}\le q<\frac{n+2}{(n-2)_+}$:}]\label{1282}
In the case $n\in\{2,3\}$ and $q$ satisfies
$$\frac{n+3}{n-1}\le q<\frac{n+2}{(n-2)_+},$$
the proof of Theorem \ref{theorem0} goes almost the same as that in the case $n\geq4$ and $\frac{n+3}{n-1}\le q<\frac{n+2}{n-2}$. Now, we sketch the proof here.

Following the lines of proof of Lemma \ref{lemma2.3*}, we obtain there exist $C_0=C_0(n,q)>0$ and $M_0=M_0(n,q)>0$,  such that  for any $\varepsilon>0$, there holds
\begin{align*}
\Delta G\ge-2\mathrm{Ric}_-G-\frac{2}{\theta}|\nabla G||\nabla\ln\omega|+C_0(v+\varepsilon)^\theta G^2-M_0\varepsilon^{q-1}G.
\end{align*}
Then, following the lines of proof of Lemma \ref{2.4*}, we obtain that,  for
\begin{align*}
t\in\left(\max\left\{\frac{8}{C_0\theta^2},1\right\},~+\infty\right),
\end{align*}
the following holds
\begin{align*}
\begin{split}
&C_0\varepsilon^\theta\int_\Omega G^{t+2}\eta^2+\frac{\mathbb{S}_\chi(M)}{2t}\left\|G^{t+1}\eta^2\right\|_{L^\chi(\Omega)}\\
\le&4\int_\Omega\mathrm{Ric}_-G^{t+1}\eta^2+2M_0\varepsilon^{q-1}\int_\Omega G^{t+1}\eta^2+\frac{12}{t}\int_\Omega G^{t+1}|\nabla\eta|^2,
\end{split}
\end{align*}
where $\eta\ge0$ and $\eta\in C_0^\infty(\Omega)$.

Once this has been done, it follows from the proof of Lemma \ref{2.5*} that, if
\begin{align*}
H(t)=\frac{\mathbb{S}_\chi (M)}{2t}-4\left\|\mathrm{Ric_-}\right\|_{L^{\frac{\chi}{\chi-1}}}>0,
\end{align*}
then there exists $C_3=C_3(n,q,t)>0$ such that
\begin{align*}
H(t)\left\{\int_{B\left(o,\frac{3}{4}R\right)}\left[(v+1)^{-\theta}v^{q-1}\right]^{(t+1)\chi}\right\}^\frac{1}{\chi}\le C_3\left(\frac{1}{R^{2(t+2)}\varepsilon^{\theta(t+1)}}+\varepsilon^{(q-1)+(q-\theta-1)(t+1)}\right)V,
\end{align*}
where $V$ denotes the volume of geodesic ball $B(o,R)$ and $\varepsilon\in(0,1)$ is any positive number.

Finally, following the lines of proof of Theorem \ref{theorem0} for the case $n\geq4$ and $\frac{n+3}{n-1}\le q<\frac{n+2}{n-2}$ we can finish the proof.\end{proof}

\begin{proof}[\bf {Proof of Theorem \ref{theorem0}:}]
Combining \thmref{main}, and the conclusions in this section,  we finish the proof of Theorem \ref{theorem0}.
\end{proof}

\section{Geometric applications}
Now we want to make use of some conclusions obtain in the previous sections to study the topological properties of a manifold which enjoy a Sobolev inequality. As mentioned in the previous, many manifolds, for instances Riemannian manifold with $\ric\geq0$ and $0 < \mathrm{AVR}_g \leq 1$, minimal hypersurfaces \cite{MS} and Ricci shrinker \cite{LW,	 WW}, support Sobolev inequalities. In particular, Brendle \cite{Bre23sobolev} obtained the sharp Sobolev inequality on complete noncompact Riemannian manifolds with nonnegative Ricci curvature using the so-called ABP method, addressing an open question posed by Cordero-Erausquin, Nazaret, and Villani \cite{CorNazVil04mass} for such manifold. Balogh and Krist\'aly \cite{Bal-Kris} provided an alternative proof of Brendle's rigidity result and confirmed that the Sobolev constant is sharp. Concretely, their results can be stated as follows:

\begin{quotation}
Let $\left(M^n, g\right)$ be a noncompact, complete $n$-dimensional Riemannian manifold with $\ric\geq0$ and $0 < \mathrm{AVR}_g \leq 1$, where $\ric$ denote the Ricci curvature of $(M, g)$. Then for all $v \in C_0^\infty\left(M^n\right)$, there holds
\begin{align*}
\norm{v}_{L^{2n/(n-2)}\left(M^n\right)} \leq \mathcal{S}\left(\mathbb{R}^n\right)\mathrm{AVR}_g^{-1/n}\norm{\nabla v}_{L^2\left(M^n\right)}.
\end{align*}
Furthermore, the constant $\mathcal{S}\left(\mathbb{R}^n\right)\mathrm{AVR}_g^{-1/n}$ is sharp.
\end{quotation}
Here, $\mathrm{AVR}_g$ is the asymptotic volume ratio of $\left(M^n, g\right)$, defined as
\begin{align*}
\mathrm{AVR}_g = \lim_{R \to \infty} \dfrac{\mathrm{Vol}(B_R(o))}{\omega_n R^n},
\end{align*}
and $\omega_n$ is the volume of the unit ball in $\mathbb{R}^n$ and $\mathcal{S}\left(\mathbb{R}^n\right)$ is the best Sobolev constant in $\mathbb{R}^n$.

Using harmonic function theory to study geometric and topological properties of manifolds has a long history, for more details we refer to \cite{LTam, LTam1, W} and references therein. Here, firstly we need to recall the definition of end in a complete noncompact manifold:
\begin{defn}
Let $(M,g)$ be a complete noncompact Riemannian manifold. We say that $M$ has $k$ ends $E_1 , \cdots , E_k$, if there exists a bounded open set $U$ so that
$$M\setminus \bar{U} =\cup_{i=1}^k E_i$$
with each $E_i$ being a noncompact connected component of $M\setminus\bar{U}$. Usually, we also call $E_1 , \cdots , E_k$ as the ends corresponding
to $U$.
\end{defn}
Denote usually the linear space spanned by bounded harmonic functions on $(M, g)$ by $H^\infty(M)$. Now let's recall some previous conclusions. The first named author of this paper has ever shown the following:

\begin{thm}[\cite{W}, Theorem 3.3]\label{thm4.1}
Let $(M, g)$ be a complete noncompact Riemannian manifold with $\dim(M)\geq 3$, Sobolev constant $\mathbb{S}_{\frac{n}{n-2}}(M)>0$ and $\ric(M) \geq 0$ outside some compact subset. Then $M$ has only finitely many ends $E_1, E_2, \cdots, E_k$ and $\dim H^\infty(M) = k$.
\end{thm}

In fact, this theorem and the Cheng-Yau's gradient estimate of positive harmonic functions (see \cite{CY}) imply that a complete noncompact Riemannian manifold, which satisfies $n=\dim(M)\geq 3$, Sobolev constant $\mathbb{S}_{\frac{n}{n-2}}(M)>0$ and $\ric(M) \geq 0$, has only an end.

The philosophy of the proof of the above theorem is: if $(M,g)$ has at least two ends and the Sobolev constant of $(M,g)$ is positive, then there exists a nonconstant, bounded, and positive harmonic function on $(M,g)$. Later on, this conclusion was also derived in \cite{CSZ}.

As an extension of \thmref{thm4.1}, we can take almost the same arguments as in \cite{DW} and \cite{W} to conclude the following theorem, i.e., \thmref{end}.

\begin{thm}\label{thm5.2}
Let $(M, g)$ be a complete noncompact Riemannian manifold with Sobolev constant $\mathbb{S}_{\chi}(M)>0$. If $(M, g)$ has at least $k$ ends, then $\dim(H^\infty(M))\geq k$.
\end{thm}

\begin{proof} Actually, by taking almost the same arguments as in \cite{DW} and \cite{W}, we can show the theorem. For the sake of completeness, here we give the routine to construct bounded positive harmonic functions on such a manifold $(M,g)$ which has at least $k$ ends and the Sobolev constant $\mathbb{S}_\chi(M)$ of $(M,g)$ is of a positive lower bound.

We let $\Omega_i$ ($i=1, 2, \cdots$) be an exhaustion of $(M, g)$, i.e., each $\Omega_i$ is an open domain contained in $M$ and $\bar{\Omega}_i$ is compact, $\bar{\Omega}_i\subset\Omega_{i+1}$ for every $i\geq 1$, and $\cup_{i=1}^\infty\Omega_i=M$. For $i \geq i_0$ and $i_0$ sufficiently large, let
$$M \setminus \Omega_i = \cup^s_{j=1}E^{(i)}_j, \quad s \geq k$$
be the disjoint union of connected components. Pick $k$ ends of $\{E^{(i)}_j: 1\leq j\leq s\}$ and denote the $k$ ends of $(M,g)$ by $E_1^{(i_0)}, \cdots, E_{k-1}^{(i_0)}$ and $E_k^{(i_0)}$, and for the sake of convenience, by $E_1, \cdots, E_k$ without confusions.

Moreover, denote $E_j^i=E_j\cap\Omega_i$. Now we consider the following Dirichlet problems of harmonic functions:
$$\Delta u_j^i=0, \quad u= 0 \,\,\mbox{on}\,\,\partial\Omega_i\setminus\partial E_j^i\quad\mbox{and}\quad u_j^i=1\,\,\mbox{on}\,\,\partial E_j^i,$$
where $j=1, \cdots, k$. Thus we can obtain $k$ sequences of harmonic functions, denoted as $\{u_j^i\}$ ($j=1, 2, \cdots, k$).

As done in \cite{DW}, we let $\phi_j(x)$ be a real smooth function on $(M,g)$ such that
$$\phi_j(x)\equiv 1\quad\mbox{on}\,\, E_j^{(i_0)}\quad\mbox{and}\quad \phi_j(x)\equiv 0\quad\mbox{in} \,\, M\setminus(\Omega_{i_0}\cup E_j^{(i_0)});$$
and then define $\rho_{j,i}$ on $\Omega_i$ by
$$\rho_{j,i}=|u^i_j-\phi_j|\quad \mbox{for} \,\, i\geq i_0.$$
It is easy to see that $|\Delta\phi_j|\in L^p$ for any $p>0$ and in the sense of distribution (see \cite{DW})
$$\Delta\rho_{j,i}\geq -|\Delta\phi_j|.$$
We choose a test function $\rho_{j,i}^{\alpha-1}$ ($\alpha>1$) to multiply the above equation and integrate on $\Omega_i$
$$\int_{\Omega_i}\rho_{j,i}^{\alpha-1}\Delta\rho_{j,i}\geq -\int_{\Omega_i}\rho_{j,i}^{\alpha-1}|\Delta\phi_j|.$$
Then we have
$$\frac{4(\alpha-1)}{\alpha^2}\int_{\Omega_i}\left|\nabla\rho_{j,i}^{\frac{\alpha}{2}}\right|^2\leq \int_{\Omega_i}\rho_{j,i}^{\alpha-1}|\Delta\phi_j|.$$
Since $(M, g)$ enjoys a Sobolev inequality as follows
\begin{equation*}
\mathbb{S}_{\chi}(M)\left(\int_M f^{2\chi}dv\right)^{\frac{1}{\chi}}\leq\int_M |\nabla f|^2 dv,
\end{equation*}
we obtain
$$\left(\int_{\Omega_i}\rho_{j,i}^{\alpha\chi}\right)^{\frac{1}{\chi}}\leq \frac{\alpha^2}{4(\alpha-1)}\mathbb{S}_\chi(M)^{-1} \int_{\Omega_i}\rho_{j,i}^{\alpha-1}|\Delta\phi_j|.$$
Applying the H\"older inequality we can see
$$\left(\int_{\Omega_i}\rho_{j,i}^{\alpha\chi}\right)^{\frac{1}{\chi}}\leq C\left(\int_{\Omega_i}\rho_{j,i}^{\alpha\chi}\right)^{\frac{\alpha-1}{\alpha\chi}}
\left(\int_{\Omega_i}|\Delta\phi_j|^\frac{\alpha\chi}{\alpha\chi-\alpha+1}\right)^\frac{\alpha\chi-\alpha+1}{\alpha\chi}.$$
It follows that
\begin{equation}\label{first}
\left(\int_{\Omega_i}\rho_{j,i}^{\alpha\chi}\right)^{\frac{1}{\alpha\chi}}\leq C \left(\int_{\Omega_i}|\Delta\phi_j|^\frac{\alpha\chi}{\alpha\chi-\alpha+1}\right)^\frac{\alpha\chi-\alpha+1}{\alpha\chi}.
\end{equation}

By the same arguments as in \cite{W} and Corollary 4.3 of \cite{DW} we can see that there exists $k$ harmonic functions $u_1, \cdots, u_{k-1}$ and $u_k$ on $(M,g)$ such that by neglecting some subsequences $u_1^i, \cdots, u_{k-1}^i$ and $u_k^i$ converge respectively in the sense of $C^k$ ($k\geq 2$) to $u_1, \cdots, u_{k-1}$ and $u_k$ on any compact subset of $(M, g)$, where $j=1, \cdots, k$, i.e.,
$$\lim_{i\to\infty} \|u^i_j-u_j\|_{C^2(\mathcal{C})}=0\quad\quad j=1,\cdots, k$$
where $\mathcal{C}\in M$ is any compact subset. In fact, we can also use directly Cheng-Yau's local gradient estimate \cite{CY} to show the convergence of these sequences $\{u^i_j\}$ ($j=1, \cdots, k;\,\, i=1, 2, \cdots$).

Now, we claim that $u_1, \cdots, u_{k-1}$ and $u_k$ are linearly independent in the linear space spanned by bounded harmonic functions on $(M, g)$, i.e., $\dim(H^\infty(M))\geq k$.

In order to prove the above assertion, firstly we note that by maximum principle there hold true $0<u_j(x)<1$ for any $x\in M$, where $j=1, 2, \cdots, k$. Then, we let $i\to \infty$ in \eqref{first}and make use of a diagonal technique to obtain
\begin{equation}\label{finn}
\left(\int_{E_j}(1-u_j)^{\alpha\chi}\right)^{\frac{1}{\alpha\chi}}<\infty\quad\quad\mbox{and}\quad\quad
\left(\int_{E_l}u_j^{\alpha\chi}\right)^{\frac{1}{\alpha\chi}}<\infty\quad \mbox{for}\,\, l\neq j,
\end{equation}
since we have
$$\int_{M}|\Delta\phi_j|^\frac{\alpha\chi}{\alpha\chi-\alpha+1}<\infty.$$
Therefore, it is easy to verify from the above integral estimates that $u_1, \cdots, u_{k-1}$ and $u_k$ are not constant functions, since $\vol(E_j)=\infty$ for each $j=1, \cdots, k$.

The remaining is to show that $u_1, \cdots, u_{k-1}$ and $u_k$ are linearly independent. Since for any geodesic ball $B_r\subset M$ we have the volume estimate
$$\mathrm{Vol}\left(B_r\right)\geq C\left(\chi, \mathbb{S}_{\chi}(M)\right)r^{\frac{2\chi}{\chi-1}},$$
this tells us that $\mathrm{Vol}(B_r(x))$ is of a uniform lower bound. Hence, noting $\chi>1$ and $\alpha>1$ we have
$$\frac{1}{\mathrm{Vol}(B_r(x))}\int_{B_r(x)}u_j\leq \left(\frac{1}{\mathrm{Vol}(B_r(x))}\int_{B_r(x)}u_j^{\alpha\chi}\right)^\frac{1}{\alpha\chi}.$$
It follows from \eqref{finn} that for $B_r(x)\subset E_i$, where $r>0$ is a fixed number, and $j\neq i$ there hold true
$$\lim_{x\in E_i\to\infty}\mathrm{Vol}(B_r(x))^{-1}\int_{B_r(x)}u_j=\lim_{x\in E_i\to\infty}\mathrm{Vol}(B_r(x))^{-\frac{1}{\alpha\chi}} \left(\int_{B_r(x)}u_j^{\alpha\chi}\right)^\frac{1}{\alpha\chi}=0$$
and
$$\lim_{x\in E_j\to\infty}\mathrm{Vol}(B_r(x))^{-1}\int_{B_r(x)}(1-u_j)=\lim_{x\in E_j\to\infty}\mathrm{Vol}(B_r(x))^{-\frac{1}{\alpha\chi}} \left(\int_{B_r(x)}(1-u_j)^{\alpha\chi}\right)^\frac{1}{\alpha\chi}=0.$$

If there holds
$$\sum_{j=1}^ka_ju_j(x)\equiv 0\quad\mbox{on}\,\, M,$$
where $a_j$ ($j=1, 2, \cdots, k$) are some real numbers, then we have
$$\lim_{x\in E_1\to\infty}\mathrm{Vol}(B_r(x))^{-1}\int_{B_r(x)}\sum_{j=1}^ka_ju_j=0,$$
which implies
$$\lim_{x\in E_1\to\infty}\mathrm{Vol}(B_r(x))^{-1}\int_{B_r(x)}a_1u_1=0.$$
We conclude that $a_1=0$. Otherwise, we have
$$\lim_{x\in E_1\to\infty}\mathrm{Vol}(B_r(x))^{-1}\int_{B_r(x)}u_1=0.$$
On the other hand, by the integral estimates \eqref{finn} we have
$$\lim_{x\in E_1\to\infty}\mathrm{Vol}(B_r(x))^{-1}\int_{B_r(x)}(1-u_1)=0.$$
Hence, it follows
$$\lim_{x\in E_1\to\infty}\mathrm{Vol}(B_r(x))^{-1}\int_{B_r(x)}(1-u_1) + \lim_{x\in E_1\to\infty}\mathrm{Vol}(B_r(x))^{-1}\int_{B_r(x)} u_1=0.$$
However, we have that for any $B_r(x)\subset E_1$ there holds
$$1=\mathrm{Vol}(B_r(x))^{-1}\int_{B_r(x)}1=\mathrm{Vol}(B_r(x))^{-1}\int_{B_r(x)}\{(1-u_1) + u_1\}$$
which implies
$$\lim_{x\in E_1\to\infty}\mathrm{Vol}(B_r(x))^{-1}\int_{B_r(x)}(1-u_1) + \lim_{x\in E_1\to\infty}\mathrm{Vol}(B_r(x))^{-1}\int_{B_r(x)} u_1=1.$$
This is a contradiction. Similarly, we can also show that $a_2=0, \cdots, a_{k-1}=0$ and $a_k=0$. This indicates that $u_1, \cdots, u_{k-1}$ and $u_k$ are linearly independent, hence we have shown
$$\dim(H^\infty(M))\geq k.$$
Thus, we finish the proof.
\end{proof}

\begin{rem}
In fact, \thmref{thm5.2} can be extended to the following case $(M, g)$ enjoys a more general Sobolev inequality:
$$\mathbb{S}_{\chi}(M)\left(\int_M f^{2\chi}dv\right)^{\frac{1}{\chi}}\leq\int_M (|\nabla f|^2 + Rf^2) dv, \quad\mbox{for any}\,\, f\in C_0^\infty(M),$$
which appears in the geometry of submanifolds and Ricci solitons (see \cite{LW, MS}). Here $R(x)\geq 0$ is usually a smooth function related to the curvature. Obviously, it is easy to verify the inequality reduces to
$$\mathbb{C}_{\chi}(M)\left(\int_M f^{2\chi}dv\right)^{\frac{1}{\chi}}\leq\int_M |\nabla f|^2  dv, \quad\mbox{for any}\,\, f\in C_0^\infty(M),$$
provided $\|R\|_{L^{\frac{\chi}{\chi-1}}}<\mathbb{S}_{\chi}(M)$. Particularly, such inequality with $\chi=\frac{n}{n-2}$ holds true in a complete minimal hypersurface in an Euclidean space (see \cite{Bre23sobolev, MS}).
\end{rem}

\medskip
\noindent{\bf Proof of \thmref{main4}}: We prove this theorem by contradiction. If $(M, g)$ has at least two ends, then, by the assumption the $\frac{n}{n-2}$-type Sobolev constant $\mathbb{S}_{\frac{n}{n-2}}(M)$ is positive and \thmref{end} we conclude that there exists a nonconstant, bounded, and positive harmonic function on $(M, g)$ (see also Theorem B in \cite{W} and Corollary 4.3 in \cite{DW} for more details). However, Corollary \ref{main2} tells us that there exists a positive constant $C(n, \beta^*)$ depending on $n$ and $\beta^*$ such that, if
$$\|\ric_-\|_{L^{\frac{n}{2}}}\leq C(n,\beta^*)\mathbb{S}_{\frac{n}{n-2}}(M),$$
then there is no nonconstant, positive harmonic function on $(M, g)$. We obtain a contradiction. We complete the proof.

\begin{thm}\label{term}
Let $(M^n, g)$ be a complete noncompact Riemannian $n$-dimensional manifold with Sobolev constant $\mathbb{S}_{\chi}(M)>0$ where $1<\chi\leq \frac{n}{n-2}$. If $\vol(B_r(x_0))=O\left(r^{\frac{2\chi}{\chi-1}}\right)$ for some geodesic ball $B_r(x_0)\subset M$, then $M$ has only finitely many ends.

In particular, suppose $(M^n, g)$ supports a $\chi$-type Sobolev inequality for some $\chi\in(1,n/(n-2))$, and assume in addition that $\|\ric_-\|_{L^{\frac{\chi}{\chi-1}}}$ is bounded. Then $M$ has only finitely many ends.
\end{thm}

\begin{proof}
The basic idea of proof is contained in the arguments on Lemma 3.2 of \cite{W}. For the sake of completeness, here we give the details. First, for any $r>0$ large enough we pick a maximal family of points $x_i\in\partial B_r(x_0)$ such that $B_{\frac{r}{4}}(x_i)$ are disjoint. Then we have
$$\partial B_r(x_0)\subset\cup_{i=1}^q\partial B_{\frac{r}{2}}(x_i),$$
where $q = \#\{x_j\}$. In fact, if there exists some $\bar{x}\in \partial B_r(x_0)\setminus \cup_{i=1}^q B_{\frac{r}{2}}(x_i)$, then $B_{\frac{r}{4}}(\bar{x})$ is disjoint with $B_{\frac{r}{4}}(x_i)$ for $i = 1, \cdots, q$. This contradicts the fact that $\{x_i\}$ is a maximal family.

Obviously we have the following inequality:
$$\sum_{i=1}^q\vol(B_{\frac{r}{4}}(x_i))\leq\vol(B_{2r}(x_0)).$$
By the assumption of this theorem, we have
$$\vol(B_{2r}(x_0))\leq C(n)r^{\frac{2\chi}{\chi-1}}.$$
By \thmref{ve} we have
$$\vol(B_{\frac{r}{4}}(x_i)) \geq C\left(\chi, \mathbb{S}_{\chi}(M)\right)r^{\frac{2\chi}{\chi-1}},$$
Hence, combining the above inequalities, we obtain
$$qC\left(\chi, \mathbb{S}_{\chi}(M)\right)\leq C(n).$$
It follows that
$$q < \left[\frac{C(n)}{C\left(\chi, \mathbb{S}_{\chi}(M)\right)}\right] + 1\equiv\Lambda,$$
where $[C(n)/C\left(\chi, \mathbb{S}_{\chi}(M)\right)]$ is the integer part of $C(n)/C\left(\chi, \mathbb{S}_{\chi}(M)\right)$. This means that the counting number of the above maximal family denoted by $\#\{x_j\}$ is bounded by a constant depending only on $n$, $\chi$ and $\mathbb{S}_{\chi}(M)$.

We claim that $M$ has only finitely many ends denoted by $E_1, \cdots, E_{k-1}$ and $E_k$ and $k\leq \Lambda$. Otherwise, we choose $r_0>0$ such that
$$M \setminus B_{r_0}(x_0) = \cup^s_{j=1}E^{(r_0)}_j, \quad s \geq \Lambda +1$$
is the disjoint union of connected components.

Now choose points \(x_j'\in \partial B_{4r_0}(x_0)\cap E^{(r_0)}_j\) for
\(j=1,\dots,s\). It is then straightforward to see that the balls
\(B_{r_0}(x_j')\) \((j=1,\dots,s)\) are pairwise disjoint. Indeed, if
\(B_{r_0}(x_i')\cap B_{r_0}(x_j')\neq\varnothing\) for some \(i\neq j\), then
\(d(x_i',x_j')<2r_0\). On the other hand, since \(x_i'\) and \(x_j'\) lie in
different connected components, any minimizing geodesic joining them must
intersect \(\partial B_{r_0}(x_0)\). Hence $
d(x_i',x_j') \ge6r_0,
$
a contradiction. Therefore the family \(\{B_{r_0}(x_j')\}_{j=1}^s\) is disjoint,
which contradicts the fact that the cardinality of maximal disjoint
family is at most \(\Lambda\). We complete the proof.
\end{proof}
\medskip\medskip

{\bf Acknowledgments:} We would like to thank Professor Jun Yang for insightful discussions. Y. Wang is supported by National Natural Science Foundation of China (Grant No.12431003); G. Wei is supported by National Natural Science Foundation of China (Grants No.12101619 and 12141106).
\bigskip
	
\bibliographystyle{acm}%siam，plain, unsrt

\begin{thebibliography}{10}
\bibitem{ADW} 	{\sc O. Agudelo, M. del Pino, J. C. Wei},
\newblock Higher-dimensional catenoid, {L}iouville equation, and
              {A}llen-{C}ahn equation.
\newblock {\em Int. Math. Res. Not. IMRN.} (2016), no. 23, 7051--7102.	
\bibitem{AK} {\sc K. Akutagawa,}
\newblock Yamabe metrics of positive scalar curvature and conformally flat manifolds.
\newblock {\em Differential Geom. Appl.} {\bf 4} (1994), no. 3, 239--258.

\bibitem{ACF}{\sc C. A. Antonini, G. Ciraolo, A. Farina,}
\newblock Interior regularity results for inhomogeneous anisotropic quasilinear equations.
\newblock{\em Math. Ann.} {\bf 387} (2023), no. 3-4, 1745--1776.

\bibitem{Bal-Kris}{\sc Z. M. Balogh, A. Krist\'{a}ly,}
\newblock Sharp isoperimetric and {S}obolev inequalities in spaces with nonnegative {R}icci curvature.
\newblock {\em Math. Ann.} {\bf 385}(2023), 1747--1773.

\bibitem{BL} {\sc H. Berestycki, P.-L. Lions,}
\newblock Nonlinear scalar field equations, I, II.
\newblock{\em Arch. Rational Mech. Anal.} {\bf 82} (1983), 313--375.

\bibitem{MR1004713}{\sc M. F. Bidaut-V\'{e}ron,}
\newblock Local and global behavior of solutions of quasilinear equations of {E}mden-{F}owler type.
\newblock {\em Arch. Rational Mech. Anal.} {\bf 107} 4 (1989), 293--324.
		
\bibitem{MR1134481}{\sc  M. F. Bidaut-V\'{e}ron, L. V\'{e}ron,}
\newblock Nonlinear elliptic equations on compact {R}iemannian manifolds and asymptotics of {E}mden equations.
\newblock{\em Invent. Math.} {\bf 106} (1991), no. 3, 489--539.

\bibitem{Bre23sobolev}{\sc S. Brendle,}
\newblock Sobolev inequalities in manifolds with nonnegative curvature.
\newblock{\em Comm. Pure Appl. Math.} {\bf 76} (2023), 2192--2218.

\bibitem{CGS}{\sc L. Caffarelli, B. Gidas, J. Spruck,}
\newblock Asymptotic symmetry and local behavior of semilinear elliptic equations with critical Sobolev growth.
\newblock{\em Comm. Pure Appl. Math.} {\bf 42} (1989), no. 3, 271--297.

\bibitem{Cai}{\sc M. L. Cai,}
\newblock Ends of Riemannian manifolds with nonnegative Ricci curvature outside a compact set.
\newblock{\em Bull. Amer. Math. Soc. (N.S.)} {\bf 24} (1991), no. 2, 371--377.

\bibitem{CCY}{\sc M. L. Cai, T. H. Colding, D. G. Yang,}
\newblock A gap theorem for ends of complete manifolds.
\newblock{\em Proc. Amer. Math. Soc.} {\bf 123} (1995), no. 1, 247--250.

\bibitem{CSZ} {\sc H. D. Cao, Y. Shen, S. H. Zhu,}
\newblock The structure of stable minimal hypersurfaces in $\mathbb{R}^n$.
\newblock{\em Math. Res. Lett.} {\bf 4} (1997), no. 5, 637--644.
\bibitem{C}{\sc G. Carron,}
\newblock In\'{e}galit\'{e}s isop\'{e}rim\'{e}triques de {F}aber-{K}rahn et cons\'{e}quences [Faber-Krahn isoperimetric inequalities and consequences].
\newblock{\em Actes de la {T}able {R}onde de {G}\'{e}om\'{e}trie {D}iff\'{e}rentielle ({L}uminy, 1992), 205--232, S\'{e}min. Congr., 1, Soc. Math. France, Paris, 1996.}

\bibitem{CM1} {\sc G. Catino, D. D. Monticelli,}
\newblock Semilinear elliptic equations on manifolds with nonnegative Ricci curvature.
\newblock{\em J. Eur. Math. Soc.(JEMS)} {\bf 28} (2026), 359-392. 	

\bibitem{MR1121147}{\sc W. X. Chen, C. M. Li,}
\newblock Classification of solutions of some nonlinear elliptic equations.
\newblock {\em Duke Math. J.} {\bf 63}  3 (1991), 615--622.

\bibitem{ChWu}{\sc Y. Z. Chen, L. C. Wu,}
\newblock Second order elliptic equations and elliptic systems. Translated from the 1991 Chinese original by Bei Hu.
\newblock{\em Translations of Mathematical Monographs,} 174. American Mathematical Society, Providence, RI, 1998. xiv+246 pp.

\bibitem{CY} {\sc S. Y. Cheng, S. T. Yau,}
\newblock Differential equations on Riemannian manifolds and their geometric applications.
\newblock{\em Comm. Pure Appl. Math.} {\bf 28} (1975), 333--354.

\bibitem{CFP} {\sc G. Ciraolo, A. Farina, C. C. Polvara,}
\newblock Classification results, rigidity theorems, and semilinear PDEs on Riemannian manifolds: a $P$-function approach.
\newblock{\em J. Eur. Math. Soc.} Doi. 10.4171/JEMS/1729.

\bibitem{CM1} {\sc T. H. Colding, W. P. II Minicozzi,}
\newblock Generalized Liouville properties of manifolds.
\newblock{\em Math. Res. Lett.} {\bf 3} (1996), no. 6, 723--729.
\bibitem{CM}{\sc T. H. Colding, W. P. II Minicozzi,}
\newblock Large scale behavior of
 kernels of Schr\"odinger operators.
\newblock{\em Amer. J. Math.} {\bf 119} (1997), no. 6, 1355--1398.
\bibitem{CM2} {\sc T. H. Colding, W. P. II Minicozzi,}
\newblock Liouville theorems for harmonic sections and applications.
\newblock{\em Comm. Pure Appl. Math.} {\bf 51} (1998), no. 2, 113--138.
\bibitem{CCKW}{\sc A. Constantin, D. G. Crowdy, V. S. Krishnamurthy, M. H. Wheeler,}
\newblock Stuart-type polar vortices on a rotating sphere.
\newblock{\em Discrete Contin. Dyn. Syst.} {\bf 41} (2021), {201--215}.

\bibitem{CG}{\sc A. Constantin, P. Germain,}
\newblock Stratospheric planetary flows from the perspective of the Euler equation on a rotating sphere.
\newblock{\em Arch. Ration. Mech. Anal.} {\bf 245} (2022), 587--644.

\bibitem{CorNazVil04mass}{\sc D. Cordero-Erausquin, B. Nazaret, C. Villani,}
\newblock A mass-transportation approach to sharp {S}obolev and {G}agliardo-{N}irenberg inequalities.
\newblock{\em Adv. Math.} {\bf 182} (2004), 307--332.

\bibitem{DWZ}{\sc X. Z.  Dai, G. F. Wei, Z. L. Zhang, }
\newblock Local Sobolev constant estimate for integral Ricci curvature bounds.
\newblock {\em Adv. Math.} {\bf 325} (2018), 1--33.

\bibitem{DKW} {\sc M. del Pino, M. Kowalczyk, J. C. Wei}
\newblock On {D}e {G}iorgi's conjecture in dimension {$N\geq 9$}.
\newblock{\em Ann. of Math.(2)} {\bf 174} (2011), 1485--1569.

\bibitem{MR0709038}{\sc  E. DiBenedetto,}
\newblock $C^{1+\alpha }$ local regularity of weak solutions of degenerate elliptic equations.
\newblock {\em Nonlinear Anal.} {\bf 7}, 8 (1983), 827--850.

\bibitem{DN1} {\sc W. Y. Ding, W. M. Ni,}
\newblock On the elliptic equation $\Delta u+Ku^{(n+2)/(n-2)}=0$ and related topics.
\newblock{\em Duke Math. J.} {\bf 52} (1985), 485--506.

\bibitem{DN2}{\sc W. Y. Ding, W. M. Ni,}
\newblock On the existence of positive entire solutions of a semilinear elliptic equation.
\newblock{\em Arch. Rational Mech. Anal.} {\bf 91} (1986), 283--308.

\bibitem{DW}{\sc W. Y. Ding, Y. D. Wang,}
\newblock Harmonic maps of complete noncompact Riemannian manifolds.
\newblock{\em Internat. J. Math.} {\bf 2} (1991), no. 6, 617--633.

\bibitem{DruHeb02AB}{\sc O. Druet, E. Hebey,}
\newblock The {$AB$} program in geometric analysis: sharp {S}obolev inequalities and related problems.
\newblock{\em Mem. Amer. Math. Soc.} {\bf 160} (2002), viii+98.


\bibitem{GHL} {\sc S. Gallot, D. Hulin, J. Lafontaine,}
\newblock Riemannian Geometry. Third edition.
\newblock{\em Universitext. Springer-Verlag, Berlin, 2004.}

\bibitem{GG1} {\sc N. Ghoussoub, C. F. Gui,}
\newblock On a conjecture of {D}e {G}iorgi and some related problems.
\newblock{\em Math. Ann.} {\bf 311} (1998), 481--491.

\bibitem{GG2} {\sc N. Ghoussoub, C. F. Gui,}
\newblock On {D}e {G}iorgi's conjecture in dimensions 4 and 5.
\newblock{\em Ann. of Math. (2)} {\bf 157} (2003), 313--334.
\bibitem{BJ}{\sc B. Gidas, J. Spruck,}
\newblock Global and local behavior of positive solutions of nonlinear elliptic equations.
\newblock{\em Comm. Pure Appl. Math.}{ \bf 34} (1981), 525--598.

\bibitem{GS}{\sc A. Grigor'yan, Y. H. Sun,}
\newblock On nonnegative solutions of the inequality {$\Delta u+u^\sigma\leq0$} on {R}iemannian manifolds.
\newblock{\em Commun. Pure. Appl. Math.} {\bf 67} (2014), no.8, 1336--1352.

\bibitem{HSW}{\sc J. He, L. L. Sun, Y. D. Wang,}
\newblock Optimal Liouville theorems for the Lane-Emden equation on Riemannian manifolds.
\newblock{\em arxiv:2411.06956.}
    	
\bibitem{HWW} {\sc J. He, Y. D. Wang, G. D. Wei,}
\newblock Gradient estimate for solutions of the equation $\Delta_pv+av^q=0$ on a complete Riemannian manifold.
\newblock{\em Math. Z.} {\bf 306} (2024), no.3, Paper Np. 42.

\bibitem{HE} {\sc E. Hebey,}
\newblock Nonlinear analysis on manifolds: Sobolev spaces and inequalities.
\newblock{\em Courant Lecture Notes in Mathematics}, {\bf 5}. 1999.

\bibitem{huang2023gradient}{\sc G. Y. Huang, Q. Guo, L. J. Guo,}
\newblock Gradient estimates for positive weak solution to $\Delta_pu+au^{\sigma}=0$ on riemannian manifolds.
\newblock {J. Math. Anal. Appl.} {\bf 533} (2024), no. 2, Paper No. 128007, 16 pp.			

\bibitem{MR2518892}{\sc B. Kotschwar, L. Ni,}
\newblock Local gradient estimates of {$p$}-harmonic functions, {$1/H$}-flow, and an entropy formula.
\newblock {\em Ann. Sci. \'{E}c. Norm. Sup\'{e}r. (4)} {\bf 42} (2009), no. 1, 1--36.

\bibitem{LTam}{\sc P. Li, L. F. Tam,}
\newblock Positive harmonic functions on complete manifolds with nonnegative curvature outside a compact set.
\newblock{Ann. of Math. (2)} {\bf 125} (1987), no. 1, 171--207.

\bibitem{LTam1}{\sc P. Li, L. F. Tam,}
\newblock Harmonic functions and the structure of complete manifolds.
\newblock{J. Differential Geom.} {\bf 35} (1992), no. 2, 359-383.
		
\bibitem{MR834612}{\sc P. Li, S. T. Yau,}
\newblock On the parabolic kernel of the {S}chr\"{o}dinger operator.
\newblock{\em Acta Math.} {\bf 156}, 3-4 (1986), 153--201.

\bibitem{LW} {\sc Y. Li, B. Wang,}
\newblock Heat kernel on Ricci shrinkers.
\newblock{\em  Calc. Var. Partial Differential Equations} {\bf 59}(2020), Paper No. 194.

\bibitem{lf} {\sc F. H. Lin,}
\newblock On the elliptic equation {$D_i[a_{ij}(x)D_jU]-k(x)U+K(x)U^p=0$}.
\newblock{\em Proc. Amer. Math. Soc.} {\bf 95} (1985), no.2, 219--226.

\bibitem{Lu}{\sc Z. H. Lu,}
\newblock Logarithmic gradient estimate and Universal bounds for semilinear elliptic equations revisited.
\newblock{\em arXiv:2308.14026}.

\bibitem{MHL} {\sc B. Q. Ma, G. Y. Huang, Y. Luo,}
\newblock Gradient estimates for a nonlinear elliptic equation on complete Riemannian manifolds,
\newblock{\it Proc. Amer. Math. Soc.} {\bf 146} (2018),  4993-5002.

\bibitem{MS} {\sc J. H. Michael, L. M. Simon.}
\newblock Sobolev and mean-value inequalities on generalized submanifolds of $R^n$.
\newblock{\it Comm. Pure Appl. Math.} {\bf 26} (1973), 361-379.

\bibitem{N} {\sc W. M. Ni,}
\newblock On the elliptic equations $\Delta u +K(x)u^{(n+2)/(n-2)}=0$, its generalizations, and applications in geometry.
\newblock{\em Indiana Univ. Math. J.} {\bf 31} (1982), 493--529.

\bibitem{MR829846}{\sc W. M. Ni, J. Serrin,}
\newblock Nonexistence theorems for singular solutions of quasilinear partial differential equations.
\newblock {\em Comm. Pure Appl. Math. }{\bf 39}, 3 (1986), 379--399.

\bibitem{O}{\sc Q. Z. Ou,}
\newblock On the classification of entire solutions to the critical p-Laplace equation.
\newblock{\em Math. Ann.} {\bf 392} (2025), no.2, 1711-1729.

\bibitem{PWW1} {\sc B. Peng, Y. D. Wang, G. D. Wei,}
\newblock Yau type gradient estimates for $\Delta u + au(\log u)^p+bu=0=0$ on Riemannian manifolds.
\newblock{\em J. Math. Anal. Appl.} {\bf 498} (2021), Paper No. 124963.

\bibitem{PW}{\sc P. Peterson, G. F. Wei,}
\newblock Relative volume comparison with integral curvature bounds.
\newblock{\em Geom. Funct. Anal.} {\bf 7} (1997), 1031-1045.

\bibitem{PW1}{\sc P. Peterson, G. F. Wei,}
\newblock Analysis and geometry on manifolds with integral Ricci curvature bounds. II.
\newblock{\em  Trans. Amer. Math. Soc. } {\bf 353} (2001), 457--478.

\bibitem{S} {\sc L. Saloff-Coste,}
\newblock Uniformly elliptic operators on Riemannian manifolds.
\newblock {\em J. Differential Geom.} {\bf 36} (1992), no.2, 417--450.

\bibitem{MR788292}{\sc R. Schoen,}
\newblock Conformal deformation of a {R}iemannian metric to constant scalar curvature.
\newblock {\em J. Differential Geom.} {\bf 20} (1984), no. 2, 479--495.

\bibitem{MR929283}{\sc R. Schoen,}
\newblock The existence of weak solutions with prescribed singular behavior for a conformally invariant scalar equation.
\newblock {\em Comm. Pure Appl. Math.} {\bf 41} (1988), no. 3, 317--392.
		
\bibitem{MR1946918}{\sc  J. Serrin, H. H. Zou,}
\newblock Cauchy-{L}iouville and universal boundedness theorems for quasilinear elliptic equations and inequalities.
\newblock {\em Acta Math.} {\bf 189} (2002), no. 1, 79--142.

\bibitem{SunW}{\sc L. L. Sun, Y. D. Wang,}
\newblock Critical quasilinear equations on Riemannian manifolds.
\newblock arXiv:2502.08495.
		
\bibitem{MR3336621}{\sc Y. H. Sun,}
\newblock On nonexistence of positive solutions of quasi-linear inequality on {R}iemannian manifolds.
\newblock {\em Proc. Amer. Math. Soc.} {\bf 143} (2015), no. 7, 2969--2984.
		
\bibitem{MR3275651}{\sc C. J.~A. Sung, J. Wang,}
\newblock Sharp gradient estimate and spectral rigidity for {$p$}-{L}aplacian.
\newblock {\em Math. Res. Lett.} {\bf 21} (2014), no. 4, 885--904.
		
\bibitem{MR0727034}{\sc  P. Tolksdorf,}
\newblock Regularity for a more general class of quasilinear elliptic equations.
\newblock {\em J. Differential Equations} {\bf 51} (1984), no. 1, 126--150.
		
\bibitem{MR0474389}{\sc K. Uhlenbeck,}
\newblock Regularity for a class of non-linear elliptic systems.
\newblock {\em Acta Math.}{\bf 138} (1977), 3-4, 219--240.

\bibitem{WZ} {\sc X. D. Wang, L. Zhang,}
\newblock Local gradient estimate for $p$-harmonic functions on Riemannian manifolds.
\newblock{\em Comm. Anal. Geom.} {\bf 19} (2011), no.4, 759--771.

\bibitem{WW} {\sc J. Wang, Y. D Wang,}
\newblock Rigidity and {$\varepsilon$}-regularity theorems of {R}icci shrinkers.
\newblock{\em Calc. Var. Partial Differential Equations}, {\bf 64} (2025), Paper No. 42, 27pp.

\bibitem{W} {\sc Y. D. Wang,}
\newblock Harmonic maps from noncompact Riemannian manifolds with non-negative Ricci curvature outside a compact set.
\newblock{\em Proc. Roy. Soc. Edinburgh Sect. A} {\bf 124} (1994), no. 6, 1259--1275.
		
\bibitem{MR4559367}{\sc Y. D. Wang, G. D. Wei,}
\newblock On the nonexistence of positive solution to {$\Delta u + au^{p+1} = 0$} on {R}iemannian manifolds.
\newblock{\em J. Differential Equations} {\bf 362} (2023), 74--87.

\bibitem{Wu}{\sc J. Y. Wu,}
\newblock Counting ends on complete smooth metric measure spaces.
\newblock{\em Proc. Amer. Math. Soc.} {\bf 144} (2016), no. 5, 2231--2239.

\bibitem{Yau}{\sc S. T. Yau,}
\newblock Harmonic functions on complete {R}iemannian manifolds.
\newblock {\em Comm. Pure Appl. Math.} {\bf 28} (1975), 201--228.

\bibitem{Z}{\sc H. C. Zhang,}
\newblock A note on {L}iouville type theorem of elliptic inequality {$\Delta u+u^\sigma\leqslant0$} on {R}iemannian manifolds.
\newblock{\em Potential Anal.} {\bf 43} (2015), no.2, 269--276.


\end{thebibliography}

\end{document}